\documentclass{article}
\pdfoutput=1
\usepackage[utf8]{inputenc}
\usepackage{amsmath,amsthm}
\usepackage[a4paper,top=3cm,bottom=3cm,right=3cm,left=3cm]{geometry}
\usepackage{natbib}
\usepackage{graphicx}
\usepackage{subcaption}
\usepackage{hyperref}

\newcommand{\Var}[1]{\mathrm{Var}\left(#1\right)}

\newcommand{\mcal}[1]{\mathcal{#1}}

\newcommand{\Ind}[1]{I\left\{ #1 \right\}}

\newtheorem{corollary}{Corollary}[section]
\newtheorem{definition}{Definition}[section]
\newtheorem{theorem}{Theorem}[section]

\newtheorem{lemma}{Lemma}[section]

\theoremstyle{remark}
\newtheorem{remark}{Remark}[section]

\title{Reinforced urns and the subdistribution beta-Stacy process prior for competing risks analysis}
\date{Accepted manuscript, Scandinavian Journal of Statistics, 2018}

\author{Andrea Arfé\footnote{Department of Decision Sciences, Bocconi University, Milan, Italy. E-mail: andrea.arfe@phd.unibocconi.it} \footnote{Department of Biostatistics and Computational Biology, Dana‐Farber Cancer Institute, Boston, Massachusetts. E-mail: aarfe@jimmy.harvard.edu} \ Stefano Peluso\footnote{Department of Statistical Sciences, Università Cattolica del Sacro Cuore, Milan, Italy.} \ Pietro Muliere\footnote{Department of Decision Sciences, Bocconi University, Milan, Italy.}}
\begin{document}
\maketitle

\begin{abstract}
In this paper we introduce the subdistribution beta-Stacy process, a novel Bayesian nonparametric process prior for subdistribution functions useful for the analysis of competing risks data. In particular, we i) characterize this process from a predictive perspective by means of an urn model with reinforcement, ii) show that it is conjugate with respect to right-censored data, and iii) highlight its relations with other prior processes for competing risks data. Additionally, we consider the subdistribution beta-Stacy process prior in a nonparametric regression model for competing risks data which, contrary to most others available in the literature, is not based on the proportional hazards assumption. 

\end{abstract}

\section{Introduction}

In the setting of clinical prognostic research with a time-to-event outcome, the occurrence of one of several competing risks may often preclude the occurrence of another event of interest \citep[Chapter 8]{Kalbfleisch2002}. In such cases it is typically of interest to assess i) the probability that one of the considered competing risks occurs within some time interval and ii) how this probability changes in association with predictors of interest  \citep{Wolbers2009,Fine1999,Putter2007}. For example, in a study of melanoma patients who received radical surgery such as that of \citet{Drzewiecki1980}, interest may be on the risk of melanoma-related mortality or melanoma-unrelated mortality and their potential predictors. Here, melanoma-related and melanoma-unrelated death act as competing risks, since onset of one necessarily precludes the onset of the other \citep[Chapter 1]{Andersen2012}.

Competing risks data has received widespread attention in the frequentist literature. It suffices to recall the comprehensive textbooks of \citet{Kalbfleisch2002}, \citet{Pintilie2006}, \citet{Aalen2008}, \citet{Lawless2011}, \citet{Andersen2012} and \citet{Crowder2012}. \citet{Putter2007}, \citet{Wolbers2009}, and \citet{Andersen2002} provide an introductory overview of standard approaches for competing risks data. Classical approaches to prediction in presence of competing risks focus on the \emph{subdistribution function}, also known as the \emph{cumulative incidence function}, which represents the probability that a specific event occurs within a given time period. \citet[Chapter 8]{Kalbfleisch2002} describe a frequentist nonparametric estimator for the subdistribution function, while \citeauthor{Fine1999a}, in their pivotal \citeyear{Fine1999a} paper, introduced a semiparametric proportional hazards model for the subdistribution function. \citet{Fine1999a} and \citet{Scheike2008} considered alternative semiparametric estimators, whilst \citet{Larson1985}, \citet{Jeong2007}, and \citet{Hinchliffe2013} considered parametric regression models for the subdistribution function. 

In contrast with the frequentist literature, the Bayesian literature on competing risks is still sparse, although several relevant contributions can be identified. \citet{Ge2012} introduced a semiparametric model for competing risks by separately modelling the subdistribution function of the primary event of interest and the conditional time-to-event distributions of the other competing risks. They modelled the baseline subdistribution hazards and the cause-specific hazards by means of a gamma process prior (see \citealp{Nieto-Barajas2002} and \citealp[Section 11.8]{Kalbfleisch2002}). \citet{DeBlasi2007} suggested a semiparametric proportional hazards regression model with logistic relative risk function for cause-specific hazards. For inference, they assign the common baseline cumulative hazard a beta process prior \citep{Hjort1990}. With the same approach, Hjort's extension of the beta process for nonhomogeneous Markov Chains \citep[Section 5]{Hjort1990} may be considered as a prior distribution on the set of cause-specific baseline hazards in a more general multiplicative hazards model (see \citealp[Chapter III]{Andersen2012} and \citealp[Chapter 9]{Lawless2011}). In the beta process for nonhomogeneous Markov Chains the individual transition hazards are necessarily independent \citep[Section 5]{Hjort1990}. The beta-Dirichlet process, a generalization of the beta process introduced by \citet{Kim2012}, relaxes this assumption by allowing for correlated hazards. \citet{Kim2012} use the beta-Dirichlet process to define a semiparametric semi-proportional transition hazards regression model for nonhomogeneous Markov Chains which, in the competing risks setting, could be used to model the cause-specific hazards. With the same purpose, \citet{Chae2013} proposed a nonparametric regression model based on a mixture of beta-Dirichlet process priors.

In this paper we introduce a novel stochastic process, a generalization of Walker and Muliere's beta-Stacy process \citep{Walker1997}, which represents a nonparametric prior distribution (i.e. a probability distribution on an infinite-dimensional space of distribution functions; see \citealp{Ferguson1973, Hjort2010, Mueller2013}) useful for the Bayesian analysis of competing risks data. This new process, which we call the \emph{subdistribution beta-Stacy} process, is conjugate with respect to right-censored observations, greatly simplifying the task of performing probabilistic predictions. We will also use the subdistribution beta-Stacy process to specify a Bayesian competing risks regression model useful for making prognostic predictions for individual patients. Contrary to most available regression approaches for competing risks, ours is not based on the proportional hazards assumption. As an illustration, we implement our model to analyse a classical dataset relating to survival of patients after surgery for malignant melanoma \citep[Chapter 1]{Andersen2012}. Throughout the paper, our perspective is Bayesian nonparametric because: i) the Bayesian interpretation of probability is especially suited for representing uncertainty when making predictions \citep{deFinetti1937, Singpurwalla1988}; ii) Bayesian nonparametric models typically provide a more honest assessment of posterior uncertainty than parametric models, as the formers are less tied to potentially restrictive and/or arbitrary parametric assumptions which may give a false sense of posterior certainty \citep{Mueller2013, Hjort2010, Phadia2015, Ghosal2017}. 

To characterize the subdistribution beta-Stacy process we adhere to the \emph{predictive approach}, a framework championed by \citet{deFinetti1937} which is receiving renewed attention in statistics and machine learning as a useful tool for constructing Bayesian nonparametric priors \citep{Fortini2012, Orbanz2015}. In the predictive approach, both the model and the prior are implicitly characterized by first specifying the predictive distribution of the observable quantities and then by appealing to results related to the celebrated de Finetti Representation Theorem \citep{Walker1999, Muliere2000, Epifani2002, Muliere2003, Bulla2007, Fortini2012}. In our context, the predictive distribution represents a specific rule prescribing how probabilistic predictions for a new patient should be performed after observing the experience of other similar (exchangeable) patients. This makes the predictive approach especially suited for our purposes: as our focus is on making prognostic predictions for individual patients, it seems natural to focus directly on the predictive distribution and its properties. Additionally, the predictive approach avoids some conceptual difficulties arising when specifying prior distributions for unobservable quantities (such as cause-specific hazards or other finite- or infinite-dimensional parameters). In fact, as often underlined by de Finetti and others, one can only express a subjective probability on observable facts; the role of unobservable quantities is just to provide a link between past experience and the probability of future observable facts (\citealp{deFinetti1937, Singpurwalla1988, Wechsler1993, Cifarelli1996}; \citealp[Chapter 4]{Bernardo2000}; \citealp{Fortini2012}).   

The predictive rule underlying the subdistribution beta-Stacy process will be described in terms of the laws determining the evolution of a \emph{reinforced urn process} \citep{Muliere2000}. Urn models have been used to characterize many common nonparametric prior processes. Classic examples include the use of a P\'olya urn for generating a Dirichlet process \citep{Blackwell1973}, P\'olya trees \citep{Mauldin1992}, and a generalised P\'olya-urn scheme for sampling the beta-Stacy process \citep{Walker1997}; \citet{Fortini2012} provide references to other modern examples. From this perspective, reinforced urn processes provide a general framework for building such urn-based characterizations. In fact, \citet{Muliere2000} and \citet{Muliere2000a} showed how reinforced urn processes can be used to characterize P\'olya trees, the beta-Stacy process, and even general neutral-to-the-right processes \citep{Doksum1974}. Reinforced urn processes have also been applied for Bayesian nonparametric inference in many contexts, from survival analysis \citep{Bulla2009} to credit risk \citep{Peluso2015}, thanks to their flexibility in modelling systems evolving through a sequence of discrete states. 

The main idea behind reinforced urn processes is that of reinforced random walk, introduced by \citet{Coppersmith1986} for modeling situations where a random walker has
a tendency to revisit familiar territory; see also \citet{Diaconis1988} and \citet{Pemantle1988,Pemantle2007}.  In detail, a reinforced urn process is a stochastic process with countable state-space $S$. Each point $x\in S$ is associated with an urn containing coloured balls. The possible colors of the ball are represented by the elements of the finite set $E$. Each urn $x\in S$ initially contains $n_x(c)\geq 0$ balls of color $c\in E$. The quantities $n_x(c)$ need not be integers, although thinking them as such simplifies the description of the process. For a fixed initial state $x_0$, recursively define the process as follows: i) if the current state is $x\in S$, then a ball is sampled from the corresponding urn and replaced together with a fixed amount $m>0$ of additional balls of the same color; hence, the extracted color is ``reinforced'', i.e. made more likely to be extracted in future draws from the same urn \citep{Coppersmith1986, Pemantle1988, Pemantle2007}; ii) if $c\in E$ is the color of the sampled ball, then the next state of the process is $q(x,c)$, where $q:S\times E\rightarrow S$ is a known function, called the \emph{law of motion} of the process, such that for every $x,y\in S$ there exists a unique $c(x,y)\in E$ satisfying $q(x,c(x,y))=y$. For our purposes, the sequence of colors extracted from the urns will represent the history of a series of sequentially observed patients. The ``reinforcement'' of colors will then correspond to the notion of ``learning from the past'' that allows predictions to be performed and which is central in the Bayesian paradigm \citep{Muliere2000,Muliere2003,Bulla2007,Peluso2015}. 

Before continuing, we must remark on the choice between continuous versus discrete time scales in the modelling of time-to-event distributions. In many, if not all, real applications, event times are not observed or available on a continuous time scale. Rather, they are either i) intrinsically discrete or ii) they are discrete because they arise from the coarsening of continuous data due to imprecise measurements (\citealp[Chapter 2]{Kalbfleisch2002}; \citealp[Chapter 1]{Tutz2016}; \citealp{Allison1982}; \citealp{Guo1994}). For this reason, throughout the paper we assume that the time axis has been pre-emptively discretized according to the fixed partition $(0,\tau_1]$, $(\tau_1,\tau_2]$, $\ldots$, $(\tau_{t-1},\tau_t]$, $\ldots$ (representing, say, successive days, months, years, etc.) implied by the measurement scale of event times in the considered application. Specifically, we assume that events can only occur at the times $\tau_1<\tau_2<\ldots$, in case (i), or that it is only possible to known in which intervals among $(0,\tau_1]$, $(\tau_1,\tau_2]$, $\ldots$, $(\tau_{t-1},\tau_t]$, $\ldots$ they occur, in case (ii). For notational simplicity, and without loss of generality, we also assume that any time-to-event variable $T>0$ takes values in the set of positive integers $t\geq 1$: the observation that $T=t$ represents either the fact that the event occurred at time $\tau_t$, in case (i), or during $(\tau_{t-1},\tau_t]$, in case (ii). 

\section{The subdistribution beta-Stacy process}\label{sec:def}

Suppose that the positive discrete random variable $T\in\{1,2,\ldots\}$ represents the time until an at-risk individual experiences some event of interest (e.g. time from surgery for melanoma to death). If the distribution of $T$ is unknown, then, in the Bayesian framework, it may be assigned a nonparametric prior to perform inference. In other words, it may be assumed that, conditionally on some random distribution function $G$ defined on $\{0,1,2,\ldots\}$, $T$ is distributed according to $G$ itself: $P(T\leq t\mid G)=G(t)$ for all $t\geq 0$, or also $P(T=t\mid G)=\Delta G(t)$, where $\Delta G(0)=G(0)=0$ and $\Delta G(t)=G(t)-G(t-1)$ for all integers $t\geq 1$. Thus the random distribution function $G$ assumes the role of an infinite-dimensional parameter, while its distribution corresponds to the nonparametric prior distribution. The \emph{beta-Stacy process} of Walker and Muliere (\citeyear{Walker1997}) is one of such nonparametric priors which has received frequent use. Specifically, a random distribution function $G$ on $\{0,1,2,\ldots\}$ is a discrete-time beta-Stacy process with parameters $\{(\beta_t,\gamma_t):t\geq 1\}$, where 
\begin{equation}\label{eqn:reccondBS}
\lim_{t\rightarrow+\infty}\prod_{u=1}^t \frac{\gamma_u}{\beta_u+\gamma_u}=0,
\end{equation}
if: i) $G(0)=0$ with probability 1 and ii) $\Delta G(t)=U_t\prod_{u=1}^{t-1}(1-U_u)$ for all $t\geq 1$, where $\{U_t:t\geq 1\}$ is a sequence of independent random variables such that $U_t\sim \textrm{Beta}(\beta_t,\gamma_t)$ for all integers $t\geq 1$. Condition (\ref{eqn:reccondBS}) is both necessary and sufficient for a random function $G(t)$ satisfying points i) and ii) to be a cumulative distribution function with probability one. The beta-Stacy process prior is conjugate with respect to right-censored data, a property that makes it especially suitable in survival analysis applications. Moreover, if $G$ is a discrete-time beta-Stacy process with parameters $\{(\beta_t,\gamma_t):t\geq 1\}$, then the predictive distribution $G^*$ of a new, yet unseen observation from $G$ is determined by $\Delta G^*(t)=E[\Delta G(t)]=\frac{\beta_t}{\beta_t+\gamma_t}\prod_{u=1}^{t-1} \frac{\gamma_u}{\beta_u+\gamma_u}$, the probability that a new observation from $G$ will be equal to $t$.

To generalize this approach to competing risks, we introduce the following definitions:

\begin{definition}\label{def:subdistf}
A function $F:\{0,1,2,\ldots\}\times\{1,\ldots,k\}\rightarrow [0,1]$, $k\geq 1$, is called a (discrete-time) \emph{subdistribution function} if it is the joint distribution function of some random vector $(T,\delta)\in\{0,1,2,\ldots\}\times\{1,\ldots,k\}$: $F(t,c)=P(T\leq t, \delta=c)$ for all $t\geq 0$ and $c\in\{1,\ldots,k\}$. A \emph{random subdistribution function} is defined as a stochastic process indexed by $\{0,1,2,\ldots\}\times\{1,\ldots,k\}$ whose sample paths form a subdistribution function almost surely. 
\end{definition}

Suppose now that $T$ represents the time until one of $k$ specific competing events occurs and that $\delta=1,\ldots,k$ indicates the type of the occurring event. For instance, for $k=2$, $T$ may represent time from surgery for melanoma to death, while $\delta$ may represent the specific cause of death: $\delta=1$ for melanoma-related mortality, $\delta=2$ for death due to other causes. As before, if the distribution of $(T,\delta)$ is unknown, then in the Bayesian nonparametric framework it is assumed that, conditionally on some random subdistribution function $F$, $(T,\delta)$ is distributed according to $F$ itself: $P(T\leq t,\delta=c\mid F)=F(t,c)$ for all $t\geq 0$ and $c=1,\ldots,k$. 

\begin{remark}
Conditionally on $F$, $\Delta F(t,c)=F(t,c)-F(t-1,c)$ is the probability of experiencing an event of type $c$ at time $t$: $\Delta F(t,c)=P(T=t,\delta=c\mid F)$. Additionally, if $G(t)=\sum_{d=1}^k F(t,d)$, $\Delta G(t)=G(t)-G(t-1)$, and $V_{t,d}=\Delta F(t,d)/\Delta G(t)$, then: $G(t)=P(T\leq t\mid F)$ is the cumulative probability of experiencing an event by time $t$, $\Delta G(t)=P(T=t\mid F)$ is the probability of experiencing an event at time $t$, and $V_{t,c}=P(\delta=c\mid T=t,F)$ is the probability of experiencing an event of type $c$ at time $t$ given that some event occurs at time $t$. Moreover, it can be shown that $F(t,c)=\sum_{u=1}^t S(u-1)\Delta A_c(u)$, where $S(t)=1-G(t)$ and $A_c(t)=\Delta F(t,c)/S(t-1)$ is the \emph{cumulative hazard} of experiencing an event of type $c$ by time $t$ \cite[Chapter 8]{Kalbfleisch2002}.
\end{remark}

To specify a suitable prior on the random subdistribution function $F$, we now introduce the subdistribution beta-Stacy process:
\begin{definition}\label{defin:subbetastacy}
Let $\{(\alpha_{t,0},\ldots,\alpha_{t,k}):t\geq 1\}$ be a collection of ($k+1$)-dimensional vectors of positive real numbers satisfying the following condition:
\begin{equation}\label{eqn:reccondSBS}
\lim_{t \rightarrow +\infty} \prod_{u=1}^t \frac{\alpha_{u,0}}{\sum_{d=0}^k \alpha_{u,d}}=0.
\end{equation}
A random subdistribution function $F$ is said to be a discrete-time \emph{subdistribution beta-Stacy} process with parameters $\{(\alpha_{t,0},\ldots,\alpha_{t,k}):t\geq 1\}$ if: 
\begin{enumerate}
\item $F(0,c)=0$ with probability 1 for all $c=1,\ldots,k$;
\item for all $c=1,\ldots,k$ and all $t\geq 1$, 
\begin{equation}\label{eqn:defsbs}
\Delta F(t,c)=W_{t,c}\prod_{u=1}^{t-1}\left(1-\sum_{d=1}^k W_{u,d}\right),
\end{equation}
with the convention that empty products are equal to $1$, and where $\{W_t=(W_{t,0},\ldots,W_{t,k}): t\geq 1\}$ is a sequence of independent random vectors such that for all $t\geq 1$, $$W_t\sim \textrm{Dirichlet}_{k+1}(\alpha_{t,0},\ldots,\alpha_{t,k}).$$
\end{enumerate}  
If in particular the $\alpha_{t,d}$ are determined as
\begin{equation*}
\alpha_{t,c} = \omega_t \Delta F_0(t,c)\quad \textrm{and} \quad \alpha_{t,0} = \omega_t \left(1-\sum_{d=1}^k F_0(t,d)\right)
\end{equation*}
for some fixed subdistribution function $F_0$ and sequence of positive real numbers $(\omega_t:t\geq 1)$, then we write $F\sim \textrm{SBS}(\omega,F_0)$.
\end{definition}

\begin{remark} \label{note:recconddef} In Section \ref{sec:predconstr}, Remark \ref{note:proofrecconddef}, it will be shown that condition (\ref{eqn:reccondSBS}) is both necessary and sufficient for a random function $F(t,c)$ satisfying points 1 and 2 of Definition \ref{defin:subbetastacy} to be a subdistribution function with probability 1. This justifies the consideration of the subdistribution beta-Stacy process as a prior distribution on the space of subdistribution functions. Also note that if $F\sim \textrm{SBS}(\omega,F_0)$, then condition (\ref{eqn:reccondSBS}) is automatically satisfied since $\sum_{d=0}^k \alpha_{t,d}= \omega_t (1-\sum_{d=1}^k F_0(t-1,d))$ and so $\prod_{t=1}^{+\infty} [\alpha_{t,0}/\sum_{d=0}^k \alpha_{t,d}]= \lim_{t\rightarrow +\infty} (1-\sum_{d=1}^k F_0(t,d))=0$, as $\lim_{t\rightarrow +\infty} \sum_{d=1}^k F_0(t,d)=1$ (provided occurrence of at least one of the $k$ events is inevitable).
\end{remark}

The following lemma (which can be proven by taking expectations of Equation (\ref{eqn:defsbs}) and using the fact that the $W_t$ are independent Dirichlet random vectors) characterizes the moments of the subdistribution beta-Stacy process.

\begin{lemma}\label{lemma:moments}
Let $F$ be a subdistribution beta-Stacy process with parameters \\ $\{(\alpha_{t,0},\ldots,\alpha_{t,k}) : t\geq 1\}$. Then
\begin{align}
E[\Delta F(t,c)]&=\frac{\alpha_{t,c}}{\sum_{d=0}^k \alpha_{t,d}}\prod_{u=1}^{t-1}\frac{\alpha_{u,0}}{\sum_{d=0}^k \alpha_{u,d}}, \label{eqn:mean}\\
E[\Delta F(t,c)^2] &= \frac{\alpha_{t,c}(1+\alpha_{t,c})}{(\sum_{d=0}^k \alpha_{t,d})(1+\sum_{d=0}^k \alpha_{t,d})}\prod_{u=1}^{t-1}\frac{\alpha_{u,0}(1+\alpha_{u,0})}{(\sum_{d=0}^k \alpha_{u,d})(1+\sum_{d=0}^k \alpha_{u,d})} \label{eqn:meansquare}\\
\Var{\Delta F(t,c)} &= E[\Delta F(t,c)]\left(\frac{1+\alpha_{t,c}}{1+\sum_{d=0}^k \alpha_{t,d}}\prod_{u=1}^{t-1}\frac{1+\alpha_{u,0}}{1+\sum_{d=0}^k \alpha_{u,d}}-E[\Delta F(t,c)]\right)\label{eqn:var}
\end{align}
for all $t\geq 1$ and $c=1,\ldots,k$.
\end{lemma}

\begin{remark}
Using Theorem 2.5 of \citet{Ng2011} it is possible to show that the vector of random probabilities 
$$\left(1-\sum_{d=1}^k \Delta F(t,d), \Delta F(t,1),\ldots, \Delta F(t,k)\right)$$
is \emph{completely neutral} in the sense of \citet{Connor1969}. Consequently, Equations (\ref{eqn:mean}) and (\ref{eqn:meansquare}) also follow from formulas (4) and (9) of \citet{Connor1969}.
\end{remark}

\begin{remark}
Note that the previous Lemma \ref{lemma:moments} also characterizes the predictive distribution associated to a subdistribution beta-Stacy process: if $F$ is a subdistribution beta-Stacy process with parameters $\{(\alpha_{t,0},\ldots,\alpha_{t,k}):t\geq 1\}$, then the predictive subdistribution function $F^*$ of a new, yet unseen observation from $F$ is determined by $\Delta F^*(t,d)=E[\Delta F(t,d)]$, which is given by Equation (\ref{eqn:mean}) ($\Delta F^*(t,d)$ is the probability that a new observation from $F$ will be equal to $(t,d)$).
\end{remark}

\begin{remark}\label{note:priorspecification}
Let $F\sim \textrm{SBS}(\omega,F_0)$. It can be shown from Equation (\ref{eqn:mean}) that $E[\Delta F(t,c)]=\Delta F_0(t,c)$ for all $t\geq 1$ and $c=1,\ldots,k$, implying both that i) $F$ is centered on $F_0$ and ii) $F_0$ is equal to the predictive distribution associated to $F$. From Equation (\ref{eqn:var}) it can be further shown that $\Var{\Delta F(t,c)}$ is a decreasing function of $\omega_t$, with $\Var{\Delta F(t,c)}\rightarrow 0$ as $\omega_t\rightarrow +\infty$ and $\Var{\Delta F(t,c)}\rightarrow F_0(t,c)(1-F_0(t,c))$ as $\omega_t\rightarrow 0$. Thus $\omega_t$ can be used to control the prior precision of the $ \textrm{SBS}(\omega,F_0)$ process.
\end{remark}

\section{Predictive characterisation of the subdistribution beta-Stacy process}\label{sec:predconstr}

\citet{Muliere2000} described a predictive construction of the discrete-time beta-Stacy process by means of a reinforced urn process $\{Y_n: n\geq 0\}$ with state space $\{0,1,2,\ldots\}$. The urns of this process contain balls of only two colors, black and white (say), and reinforcement is performed by the addition of a single ball ($m=1$). To intuitively describe this process, suppose that each patient in a series is observed from an initial time point until the onset of an event of interest. The process $\{Y_n: n\geq 0\}$ starts from $Y_0=0$, signifying the start of the observation for the first patient, and then evolves as follows: if $Y_n=t$ and a black ball is extracted, then the current patient does not experience the event at time $t$ and $Y_{n+1}=t+1$; if instead a white ball is extracted, then the current patient experiences the event at time $t$ and $Y_{n+1}=0$, so the process is restarted to signify the start of the observation of a new patient. With this interpretation, the number $T_n$ of states visited by $\{Y_n: n\geq 0\}$ between the $(n-1)$-th and $n$-th visits to the initial state 0 correspond to the time of event onset for the $n$-th patient. If the process $\{Y_n:n\geq 0\}$ is recurrent (so the times $T_n$ are almost surely finite), a representation theorem for reinforced urn processes implies that the process $\{Y_n:n\geq 0\}$ is a mixture of Markov Chains. The corresponding mixing measure is such that the rows of the transition matrix are independent Dirichlet processes (\citealp[Theorem 2.16]{Muliere2000}; see \citealp{Ferguson1973} for the definition of a Dirichlet process). Using this representation, \citet{Muliere2000} showed that the sequence $\{T_n:n\geq 1\}$ is exchangeable and that there exists a random distribution function $G$ such that i) conditionally on $G$, the times $T_1, T_2, \ldots$ are i.i.d. with common distribution function $G$, and ii) $G$ is a beta-Stacy process \citep[Section 3]{Muliere2000}.

In this section, we will generalize the predictive construction of \citet{Muliere2000} to yield a similar characterization of the subdistribution  beta-Stacy process. To do so, consider a reinforced urn process $\{X_n:n\geq 0\}$ with state space $S=\{0,1,2,\ldots\}\times E$, set of colors $E=\{0,1,\ldots,k\}$ ($k\geq 1$), starting point $X_0=(0,0)$, and law of motion defined by $q((t,0),c)=(t+1,c)$ and $q((t,d),c)=(0,0)$ for all for all integers $t\geq 0$ and $c,d=0,1,\ldots,k$, $d\neq 0$. Further suppose, for simplicity of presentation, that reinforcement is performed by the addition of a single ball ($m=1$) as before (but see Remark \ref{remark:m} below for the case where $m\in(0,+\infty)$). The initial composition of the urns is given as follows: i) $n_{(t,0)}(c)=\alpha_{t+1,c}$ for all integers $t\geq 0$ and $c=0,1,\ldots,k$; ii) $n_{(t,d)}(0)=1$, $n_{(t,d)}(c)=0$ for all integers $t\geq 0$ and $c,d=1,\ldots,k$, $d\neq 0$. Now,  define $\tau_0=0$ and $\tau_{n+1}=\inf\{t>\tau_n: X_t=(0,0)\}$ for all integers $n\geq 0$. The process $\{X_n:n\geq 0\}$ is said to be \emph{recurrent} if $P(\cap_{n=1}^{+\infty}\{\tau_n<+\infty\})=1$. Additionally, let $T((t,c))=t$ and $D((t,c))=c$ for all $(t,c)\in S$. For all $n\geq 1$, set $T_n=T(X_{\tau_n-1})$, the length of the sequence of states between the $(n-1)$-th and the $n$-th visits to the initial state $(0,0)$, and $D_n=D(X_{\tau_n-1})$, the color of the last ball extracted before the $n$-th visit to $(0,0)$. 

The process $\{X_n:n\geq 0\}$ can be interpreted as follows: a patient initially at risk of experiencing any of $k$ possible outcomes is followed in time starting from time $t=0$; at each time point $t$, the color of the extracted ball represents the status of the patient at the next time point $t+1$; if a ball of color $0$ is extracted, the patient remains at risk at the next time point; if instead a ball of color $c\in\{1,\ldots,k\}$ is extracted, then the patient will experience an outcome of type $c$ at the next time point. The process returns to the initial state after such an occurrence to signify the arrival of a new patient. With this interpretation, the variable $T_n$ represents the time at which the $n$-th patient experiences one of the $k$ events under study, while $D_n$ encodes the type of the realized outcome. These concepts are illustrated in Figure \ref{fig:urns}. Moreover, note that, although slightly different, the reinforced urn process used to construct the beta-Stacy process by \citet{Muliere2000} is essentially equivalent to the process $\{X_n:n\geq 0\}$ in the particular case where $k=1$, with color 0 being black and color 1 being white in the above description.

\begin{figure}
\centering
\includegraphics[scale=0.53,trim={0 6cm 0 0}]{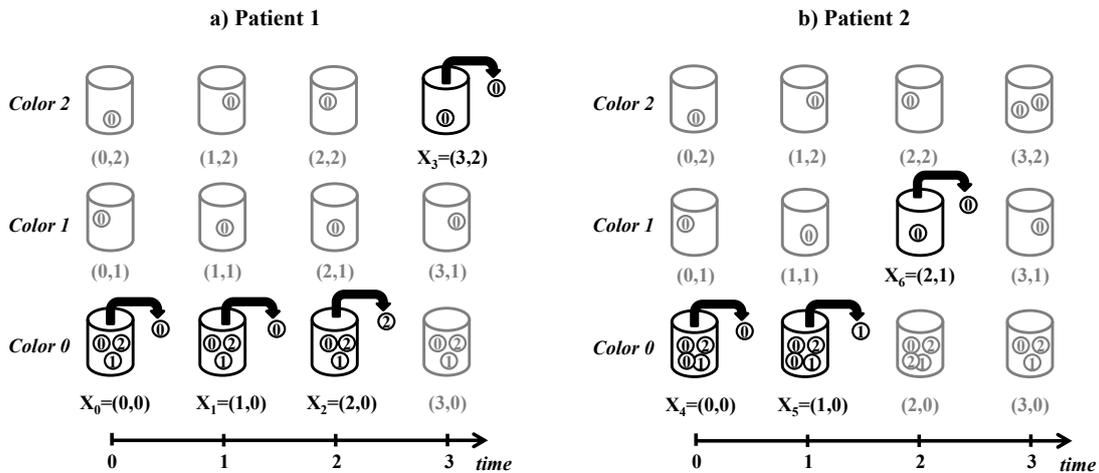}
\caption{Illustration of the reinforced urn process characterizing the subdistribution beta-Stacy process assuming $k=2$. In both panels, the horizontal axis measures the time since the last visit to the urn representing the state $(0,0)$. The process starts from the $(0,0)$ urn in Panel a, in which all urns are represented at their initial composition. In this example, balls of colors 0, 0, and 2 are successively extracted from the urns visited by the process, respectively at times 0, 1, and 2. At time 3 the process visits the $(3,2)$ urn, from which only balls of color 0 can be extracted. The process then returns to the $(0,0)$ urn and continues as shown in Panel b, where the composition of the urns has been updated by reinforcement. Suppose now that each visit to $(0,0)$ represents the arrival of a new melanoma patient at the moment of surgery. If color 1 represents death due to melanoma and color 2 represents death due to other causes, then the sequence of urns visited in Panel a corresponds to the history of an individual (Patient 1) who dies of causes not related to melanoma after 3 time instants since surgery ($T_1=3$, $D_1=2$), while Panel b represents the history of a subsequently observed individual (Patient 2) who dies due to melanoma after 2 time instants since surgery ($T_2=2$, $D_2=1$).}
\label{fig:urns}
\end{figure}

Continuing, in accordance with \citet{Diaconis1980} we say that the process $\{X_n:n\geq 0\}$ is \emph{Markov exchangeable} if $P(X_0=x_0,\ldots,X_n=x_n)=P(X_0=y_0,\ldots,X_n=y_n)$ for all finite sequences $(x_0,\ldots,x_n)$ and $(y_0,\ldots,y_n)$ of elements of $S$ such that i) $x_0=y_0$ and ii) for any $s_1,s_2\in S$, the number of transitions from $s_1$ to $s_2$  is the same in both sequences.

\begin{lemma}\label{lemma:partexch}
The process $\{X_n:n\geq 0\}$ is Markov exchangeable. Consequently, if $\{X_n:n\geq 0\}$ is recurrent, then it is also a mixture of Markov Chains with state space $S$. In other words, there exists a probability measure $\mu$ on the space $\mcal{M}$ of all transition matrices on $S\times S$ and a $\mcal{M}$-valued random element $\Pi\sim\mu$ such that for all $n\geq 1$ and all sequences $x_0,\ldots,x_n\in S$ with $x_0=(0,0)$, 
$$P(X_0=x_0,\ldots,X_n=x_n\mid \Pi)=\prod_{i=0}^{n-1} \Pi(x_i,x_{i+1}),$$
where $\Pi(x,y)$ is the element on the $x$-row and $y$-th column of $\Pi$. Additionally, for each $x=(t,c)\in S$, let $\mcal{N}_x(\cdot)$ be the measure on $S$ (together with the Borel $\sigma$-algebra generated by the discrete topology) which gives mass $n_{(t,c)}(d)$ to $q((t,c),d)$ for all $d=0,1,\ldots,k$, and null mass to all other points in $S$. Then, the random probability measure $\Pi(x,\cdot)$ on $S$ is a Dirichlet process with parameter measure $\mcal{N}_x(\cdot)$.
\end{lemma}
\begin{proof}
The thesis follows immediately from Theorem 2.3 and 2.16 of \citet{Muliere2000} and Theorem 7 of \citet{Diaconis1980}.
\end{proof}

\begin{lemma}\label{lemma:recurrence}
The process $\{X_n:n\geq 0\}$ is recurrent if and only if $\{(\alpha_{t,0},\ldots,\alpha_{t,k}):t\geq 1\}$ satisfies condition (\ref{eqn:reccondSBS}).
\end{lemma}
\begin{proof}
First observe that
\begin{align*}
P(\tau_1=+\infty) &= \lim_{n\rightarrow +\infty} P(\tau_1>n) \\
&=\lim_{n\rightarrow +\infty} P(X_0=(0,0),X_1=(1,0),\ldots,X_{n-1}=(n-1,0)) \\
&= \lim_{n\rightarrow +\infty} \prod_{t=0}^{n-1} \frac{n_{(t,0)}(0)}{\sum_{d=1}^k n_{(t,0)}(d)} \\
&=\lim_{n\rightarrow +\infty} \prod_{t=1}^n \frac{\alpha_{t,0}}{\sum_{d=1}^k \alpha_{t,d}}.
\end{align*}

Consequently, if $\{X_n:n\geq 0\}$ is recurrent, then $P(\tau_1=\infty)=0$ and so condition (\ref{eqn:reccondSBS}) must hold. Conversely, suppose that condition (\ref{eqn:reccondSBS}) is satisfied. Then $P(\tau_1<+\infty)=1$. By induction on $n\geq 1$, suppose that $P(\cap_{i=1}^n\{\tau_i<+\infty\})=1$. Then
$$P(\tau_{n+1}=+\infty)=\int_{\cap_{i=1}^n\{\tau_i<+\infty\}} P(\tau_{n+1}=+\infty\mid T_1,\ldots,T_n)dP.$$
Given $T_1,\ldots,T_n$, if $\tau_{n+1}=+\infty$ then the process must visit all states $(t,0)$ with $t\geq 0$ starting from time $\tau_n$. Since the states $(t,0)$ for $t>L:=\max(T_1,\ldots,T_n)+1$ correspond to previously unvisited urns, the probability of this event is bounded above by 
$$\lim_{n\rightarrow +\infty} \prod_{i=L}^n \frac{n_{(i,0)}(0)}{\sum_{d=1}^k n_{(i,0)}(d)}=\lim_{n\rightarrow +\infty} \prod_{i=L+1}^n \frac{\alpha_{i,0}}{\sum_{d=1}^k \alpha_{i,d}}.$$
Hence
\begin{align*}
P(\tau_{n+1}=+\infty)\leq \int_{\cap_{i=1}^n\{\tau_i<+\infty\}} \lim_{n\rightarrow +\infty} \prod_{i=L+1}^n \frac{\alpha_{i,0}}{\sum_{d=1}^k \alpha_{i,d}} dP = 0,
\end{align*}
where the last equality follows from condition (\ref{eqn:reccondSBS}). Consequently, $P(\cap_{i=1}^{n+1}\{\tau_i<+\infty\})=1$. This argument shows that $P(\cap_{i=1}^{+\infty}\{\tau_i<+\infty\})=1$ and so the process must be recurrent, as needed.
\end{proof}

\begin{theorem}\label{thm:reprthm}
Suppose that the process $\{X_n:n\geq 0\}$ is recurrent. Then there exists a random subdistribution function $F$, such that, given $F$, the $(T_n,D_n)$ are i.i.d. distributed according to $F$. Moreover, i) $F$ is determined as a function of the random transition matrix $\Pi$ from Lemma \ref{lemma:partexch}, and ii) $F$ is a subdistribution beta-Stacy process with parameters $\{(\alpha_{t,0},\ldots,\alpha_{t,k}):t\geq 1\}$.
\end{theorem}
\begin{proof}
Let $\Pi$ be the random transition matrix on $S\times S$ provided by Lemma \ref{lemma:partexch} and define $F(t,c)=P(T_1\leq t, D_1=c\mid \Pi)$, which is clearly a random subdistribution function. Moreover, for all $c=1,\ldots,k$, 
$$F(0,c)=P(T_1=0,D_1=c\mid \Pi)\leq P(T(X_{\tau_1-1})=1\mid \Pi)= P(\tau_1=1\mid \Pi)=0.$$
Instead, for all $c=1,\ldots,k$ and all $t\geq 1$,
\begin{align*}
\Delta F(t,c)&=P(T_1= t, D_1=c\mid \Pi) \\
&= P(X_0=(0,0),\ldots,X_{t-1}=(t-1,0),X_{t}=(t,c)\mid \Pi) \\
&= \Pi((t-1,0),(t,c))\prod_{u=0}^{t-2}\Pi((u,0),(u+1,0)).
\end{align*}
Now, for all $t\geq 1$ and $d=0,1,\ldots,k$, 
$$\mcal{N}_{(t-1,0)}(\{(t,d)\})=\mcal{N}_{(t-1,0)}(\{q((t-1,0),d)\})=n_{(t-1,0)}(d)=\alpha_{t,d}.$$ 
Then, from Lemma \ref{lemma:partexch} again and from the properties of the Dirichlet process \citep{Ferguson1973}, for all $t\geq 1$,
$\big{(}\Pi((t-1,0),(t,0)),\ldots,\Pi((t-1,0),(t,k))\big{)} 
\sim \textrm{Dirichlet}_{k+1}\left(\alpha_{t,0},\ldots,\alpha_{t,k}\right)$.
Hence, Lemma \ref{lemma:recurrence} implies that $F$ is subdistribution beta-Stacy with parameters $\{(\alpha_{t,0},\ldots,\alpha_{t,k}):t\geq 1\}$.

To show that, given $F$, the $(T_n,D_n)$ are i.i.d. distributed according to $F$, it suffices to note that for all $(t_1,d_1),\ldots,(t_n,d_n)\in S$ such that $t_i\geq 1$ for all $i=1,\ldots,n$,  it holds that
\begin{align*}
P&\left((T_1,D_1)=(t_1,d_1),\ldots,(T_n,D_n)=(t_n,d_n)\mid \Pi\right)  \\
&= \prod_{i=1}^{n} \left\{\Pi((t_i-1,0),(t_i,d_i))\prod_{t=0}^{t_i-1}\Pi((t,0),(t+1,0))\right\} \\
&= \prod_{i=1}^{n} \Delta F(t_i,d_i).
\end{align*}
Since $F$ is a function of $\Pi$, this concludes the proof.
\end{proof}

\begin{remark}\label{note:proofrecconddef}
Suppose that  $F$  is a random function satisfying points 1 and 2 of Definition \ref{defin:subbetastacy}. The proof of Theorem \ref{thm:reprthm} also shows that, if condition  (\ref{eqn:reccondSBS}) is satisfied, then $F$ is a random subdistribution function. This is because condition  (\ref{eqn:reccondSBS}) coincides with the recurrency condition in Lemma  \ref{lemma:recurrence}. Suppose instead that $F$ is a subdistribution function with probability 1. Then $\tilde{F}(t,c)=E[F(t,c)]$ is a subdistribution function and 
$$P(T_1\leq t, D_1=c)=\tilde{F}(t,c)=\frac{\alpha_{t,d}}{\sum_{c=0}^k \alpha_{t,c}} \prod_{u=1}^{t-1} \frac{\alpha_{u,0}}{\sum_{c=0}^k \alpha_{u,c}}$$
for all $t\geq 0$ and $c=1,\ldots,k$. Hence it must be
$$0=P(T_1=+\infty)=\lim_{t\rightarrow +\infty}P(X_0=(0,0),\ldots,X_t=(t,0))=\lim_{t\rightarrow+\infty}\prod_{u=1}^t \frac{\alpha_{u,0}}{\sum_{c=0}^k \alpha_{u,c}}.$$
Thus condition (\ref{eqn:reccondSBS}) must hold. Therefore, condition (\ref{eqn:reccondSBS}) is both necessary and sufficient for $F$ to be a random subdistribution function, justifiying the claim anticipated in Remark \ref{note:recconddef}.
\end{remark}

Another immediate consequence of Theorem \ref{thm:reprthm} is the following corollary:

\begin{corollary}
The sequence of random variables $\{(T_n,D_n):n\geq 1\}$ induced by the reinforced urn process $\{X_n:n\geq 0\}$ is exchangeable.
\end{corollary} 

This fact could also have been proven directly through an argument similar to that at the end of Section 2 of \citet{Muliere2000}. To elaborate, suppose that $\{Y_n: n\geq 0\}$ is a recurrent stochastic process with countable state space $S$ and such that $X_0=x_0\in S$ with probability one. Then a \emph{$x_0$-block} is defined as any finite sequence of states visited by process which begins from $x_0$ and ends at the state immediately preceding the successive visit to $x_0$. \citet{Diaconis1980} showed that if $\{Y_n: n\geq 0\}$ is also Markov exchangeable, then the sequence $\{B_n: n\geq 1\}$ of its $x_0$-blocks is exchangeable. Now, consider the reinforced urn process $\{Y_n: n\geq 0\}$ used by \citet{Muliere2000} for constructing the beta-Stacy process and described at the beginning of this section. This process is Markov exchangeable and so, under a recurrency condition, its sequence of $0$-blocks $\{B_n: n\geq 1\}$ is exchangeable. Consequently, so must be the corresponding sequence of total survival times $\{T_n=f(B_n): n\geq 1\}$, where $f(B)$ is the length of the $0$-block $B$ after excluding its initial element. Each $0$-block $B_n$ must have the form $(0,1,\ldots,t)$ for some $t\geq 1$ and $f((0,1,\ldots,t))=t$ for all $t\geq 1$.

In our setting, it can easily be seen that the $(0,0)$-blocks of the reinforced urn process $\{X_n:n\geq 0\}$ introduced in this section are finite sequences of states of the form $((0,0),(1,0),\ldots,(t-1,0),(t,c))$ for some $t\geq 1$ and $c=1,\ldots,k$. Any such $(0,0)$-block represents the entire observed history of an individual at risk of developing any one of the $k$ considered competing risks. For example, the history of Patient 1 in Figure \ref{fig:urns}(a) is represented by the $(0,0)$-block $B_1=((0,0),(1,0),(2,0),(3,2))$, while that of Patient 2 in Figure \ref{fig:urns}(b) is represented by the $(0,0)$-block $B_2=((0,0),(1,0),(2,1))$. If $\{X_n:n\geq 0\}$ is recurrent, by Lemma \ref{lemma:partexch} its sequence of $(0,0)$-blocks $\{B_n: n\geq 1\}$ is exchangeable. Hence, so must be the sequence $\{(T_n,D_n)=f(B_n):n\geq 1\}$, as claimed, where $f(B)$ is the last state in the $(0,0)$-block $B$. For the example in Figure \ref{fig:urns}, $f(B_1)=(T_1,D_1)=(3,2)$ and $f(B_2)=(T_2,D_2)=(2,1)$. 

\begin{remark}\label{remark:m}
Throughout this section, we have assumed for simplicity that each extracted ball is reinforced by only by single ball of the same color, i.e. $m=1$. In general, a number $m>0$ could be considered. It is possible to show (see for example \citealp{Amerio2004} or \citealp{Mezzetti2007}) that Theorem \ref{thm:reprthm} would still hold with $F$ distributed according a subdistribution beta-Stacy process with parameters $\{(\alpha_{t,0}/m,\ldots,\alpha_{t,k}/m):t\geq 1\}$. In particular, if $\alpha_{t,c} = \omega_t \Delta F_0(t,c)$ and $\alpha_{t,0} = \omega_t (1-\sum_{d=1}^k F_0(t,d))$, then $F\sim \textrm{SBS}(\omega/m,F_0)$. Hence, the number of balls $m$ used for reinforcement can be used to control concentration of the prior around its mean.  
\end{remark}


\section{Posterior distributions and censoring}\label{sec:post}

Suppose that $(T_i,D_i)$ is distributed according to some subdistribution function $F$ and $T_i>0$ with probability 1 for all $i=1,\ldots,n$. If the value $(T_i,D_i)$ can be potentially right-censored at the known time $c_i\in\{0,1,2,\ldots\}\cup\{+\infty\}$, then instead of observing the actual value $(T_i,D_i)$ one is only able to observe $(T_i^*,D_i^*)$, where $(T_i^*,D_i^*)=(T_i,D_i)$ if $T_i\leq c_i$ and $(T_i^*,D_i^*)=(c_i,0)$ if $T_i>c_i$ (if $c_i=+\infty$, then $(T_i,D_i)$ is not affected by censoring). The following theorem shows that the subdistribution beta-Stacy process has a useful conjugacy property even in presence of such right-censoring mechanism. 

\begin{theorem}\label{thm:conjugacy}
Suppose that $(T_1,D_1)$, $\ldots$, $(T_n,D_n)$ is an i.i.d. sample from a subdistribution function $F$ distributed as a subdistribution beta-Stacy process with parameters $\{(\alpha_{t,0},\ldots,\alpha_{t,k}):t\geq 1\}$. If $(T_1,D_1),\ldots,(T_n,D_n)$ are potentially right-censored at the known times $c_1,\ldots,c_n$, respectively, then the posterior distribution of $F$ given $(T_1^*,D_1^*)$,$\ldots$, $(T_n^*,D_n^*)$  is a subdistribution beta-Stacy with parameters $\{(\alpha_{t,0}^*,\ldots,\alpha_{t,k}^*):t\geq 1\}$, where $\alpha_{t,0}^*=\alpha_{t,0}+l_t+m_{t,0}$, $\alpha_{t,d}^*=\alpha_{t,d}+ m_{t,d}$ for all integers $t\geq 1$ and for $d=1,\ldots,k$,
where $l_t=\sum_{i=1}^n \Ind{T_i^*>t}$ and $m_{t,d}=\sum_{i=1}^n \Ind{T_i^*=t,D_i=d}$ for all $t\geq 1$ and $d=0,1,\ldots,k$.
\end{theorem}
\begin{proof}
To prove the thesis, it suffices it is true for $n=1$, as the general case will then follow from an immediate induction argument. To do so, first note that, with reference to the renforced urn process $\{X_n: n\geq 0\}$ of Section \ref{sec:predconstr}, condition (\ref{eqn:reccondSBS}) implies that $F$ can be seen as a function of some random transition matrix $\Pi$ as in the proof of Theorem \ref{thm:reprthm}. Assume now that $(T_1^*,D_1^*)=(t,d)$ for some $t\geq 1$ and $d=0,1,\ldots,k$. Since observing $(T_1^*,D_1^*)$ is equivalent to observing $X_0=(0,0),\ldots,X_{t-1}=(t-1,0),X_t=(t,d)$, Corollary 2.21 of \citet{Muliere2000} implies that, conditionally on $(T_1^*,D_1^*)=(t,d)$, the rows of $\Pi$ are independent and, for all $x\in S$, the parameter measure of the $x$-th row of $\Pi$ assigns mass $n_{(0,0)}(0)+1,\ldots,n_{(t-2,0)}(0)+1, n_{(t-1,0)}(d)+1$ to the states $(1,0),\ldots,(t-1,0),(t,d)$, respectively, and mass $n_{(t',d')}(c)$ to all other states $q((t',d'),c)\neq (1,0)$,$\ldots$, $(t-1,0),(t,d)$ in $S$. Since $\alpha_{t,d}=n_{(t-1,0)}(d)$ for all $t\geq 1$ and $d=0,1,\ldots,k$, it can now be seen that, conditionally on $(T_1^*,D_1^*)$, $F$ must be subdistribution beta-Stacy with parameters $\{(\alpha_{t,0}^*,\ldots,\alpha_{t,k}^*):t\geq 1\}$ defined by $\alpha_{t,0}^*=\alpha_{t,0}+\Ind{T_1^*>t}+\Ind{T_1^*=t, D_1^*=0}$, $\alpha_{t,d}^*=\alpha_{t,d}+\Ind{T_1^*=t, D_1^*=d}$ for all integers $t\geq 1$ for $d=1,\ldots,k$. 
\end{proof}

The following corollary is now a direct consequence of Equation (\ref{eqn:mean}) in Lemma \ref{lemma:moments}.

\begin{corollary}\label{note:preddist}
The predictive distribution $F^*(t,d)$ of a new (non-censored) observation $(T_{n+1},$ $D_{n+1})$ from $F$ having previously observed $(T_1^*,D_1^*),\ldots,(T_n^*,D_n^*)$ is determined by
\begin{align*}
\Delta F^*(t,d) &= P((T_{n+1},D_{n+1})=(t,d)\mid (T_1^*,D_1^*),\ldots,(T_n^*,D_n^*)) \\
&= E\left[\Delta F(t,d)\mid (T_1^*,D_1^*),\ldots,(T_n^*,D_n^*)\right] \\
&= \frac{\alpha_{t,d}^*}{\sum_{c=0}^k \alpha^*_{t,c}}\prod_{u=1}^{t-1} \frac{\alpha^*_{u,0}}{\sum_{c=0}^k \alpha^*_{u,c}}.
\end{align*}
for all $t\geq 1$ and $d=1,\ldots,k$.
\end{corollary}

The following result instead follows from Corollary \ref{note:preddist} and Remark \ref{note:priorspecification}.

\begin{corollary}\label{note:postsbs}
Assume that $F\sim \textrm{SBS}(\omega,F_0)$ a priori. Then, the posterior distribution of $F$ given the observed values of $(T_1^*,D_1^*),\ldots,(T_n^*,D_n^*)$ is $\textrm{SBS}(\omega^*,F^*)$, where 
\begin{align*}
F^*(t,c)& =\sum_{u=1}^t S^*(u-1)\Delta A^*_c(u), \\
A^*_c(t)&=\sum_{u=1}^t \frac{\omega_u\Delta F_0(u,c)+m_{u,c}}{\omega_u(1-\sum_{d=1}^k F_0(u-1,d))+l_u+\sum_{d=0}^k m_{u,d}},\\
S^*(t)&=\prod_{u=1}^t\left(1-\frac{\omega_u\sum_{d=1}^k \Delta F_0(u,d)+\sum_{d=1}^k m_{u,d}}{\omega_u(1-\sum_{d=1}^k F_0(u-1,d))+l_u+\sum_{d=0}^k m_{u,d}}\right),
\end{align*}
and 
$$\omega_t^*=\frac{\omega_t\left[1-\sum_{d=1}^k F_0(t,d)\right]+l_t+m_{t,0}}{1-\sum_{d=1}^k F^*(t,d)}.$$ 
\end{corollary}

\begin{remark} \label{note:preddistconv}
As $\max_{u=1,\ldots,t}(\omega_u)\rightarrow 0$, $S^*(t)$ converges to the discrete-time Kaplan-Meier estimate 
$\widehat{S}(t)=\prod_{u=1}^t (1-[\sum_{d=1}^k m_{u,d}]/[l_u+\sum_{d=0}^k m_{u,d}]),$ while $A^*_c(t)$ converges to the Nelson-Aalen estimate $\widehat{A}_c(t)=\sum_{u=1}^t m_{u,c}/(l_u+\sum_{d=0}^k m_{u,d})$ for all times $t\geq 1$ for which $\widehat{S}(t)$ and the $\widehat{A}_c(t)$ are defined (i.e. such that $l_t+\sum_{d=0}^k m_{t,d}]>0$). All in all, $F^*(t,c)$, which coincides with the optimal Bayesian estimate of $F$ under a squared-error loss, converges to $\widehat{F}(t,c)=\sum_{u=1}^t \widehat{S}(u-1)\Delta \widehat{A}_c(u)$, the classical non-parametric estimate of $F(t,c)$ of \citet[Chapter 8]{Kalbfleisch2002}, for all times $t\geq 1$ for which this is defined. Conversely, if $\min_{u=1,\ldots,t}(\omega_u)\rightarrow +\infty$, then $S^*(t)$ converges to $1-\sum_{d=1}^k F_0(t,d)$, $A_c(t)$ converges to the corresponding cumulative hazard of $F_0$, and therefore $F^*(t,c)$ converges to the prior mean $F_0(t,c)$ for all times $t\geq 1$ and $c=1,\ldots,k$. 
\end{remark}

\begin{remark} (Censored data likelihood)\label{note:censlik}
Given a sample $(t_1^*,d_1^*)$,$\ldots$, $(t_n^*,d_n^*)$ of censored observations from a subdistribution function $F(t,c)$, define $z_i=\Ind{d_i^*\neq 0}$ for all $i=1,\ldots,n$. It can then be shown that the likelihood function for $F$ is  
\begin{equation}\label{eqn:likelihood}
\begin{aligned}
L(F) &= P\left((T_1^*,D_1^*)=(t_1^*,d_1^*),\ldots,(T_n^*,D_n^*)=(t_n^*,d_n^*)\mid F\right) \\
&= \prod_{i=1}^n \Delta F(t_i^*,d_i^*)^{z_i} \left[1-\sum_{d=1}^k F(t^*_i,d)\right]^{1-z_i}.
\end{aligned}
\end{equation}
\end{remark}

So far the censoring times $c_1,\ldots,c_n$ have been considered fixed and known. Theorem \ref{thm:conjugacy} however continues to hold also in the following more general setting in which censoring times are random: let the censored data be defined as $T_i^*=\min(T_i,C_i)$ and $D_i^*=\Ind{T_i\leq C_i}$ for all $i=1,\ldots,n$, where i) $C_1,\ldots,C_n$ are independent random variable with common distribution function $H(t)$, ii) conditional on $F$ and $H$, $(T_1,D_1),\ldots,(T_n,D_n)$ and $C_1,\ldots,C_n$ are independent, and iii) F and H are a priori independent. Adapting the terminology of Heitjan and Rubin (\citeyear{Heitjan1991,Heitjan1993}), in this case the random censoring mechanism is said to be \emph{ignorable}.

\begin{theorem}\label{thm:censoringgeneral}
If censoring is random and ignorable and $F$ is a priori a subdistribution beta-Stacy process, then  the marginal likelihood for $F$ is proportional to the likelihood $L(F)$ defined in Equation (\ref{eqn:likelihood}). Consequently, the posterior distribution of $F$ given $(T_1^*,D_1^*)$, $\ldots$, $(T_n^*,D_n^*)$ is the same as that described in Theorem \ref{thm:conjugacy}. 
\end{theorem}
\begin{proof}
The likelihood function for $F$ and $H$ given a sample $(t_1^*,d_1^*)$,$\ldots$, $(t_n^*,d_n^*)$ of observations affected from ignorable random censoring is 
\begin{align*}
L^*(F,H) &= P\left((T_1^*,D_1^*)=(t_1^*,d_1^*),\ldots,(T_n^*,D_n^*)=(t_n^*,d_n^*)\mid F,H\right) \\
&= L(F) \prod_{i=1}^n \Delta H(t_i^*)^{1-z_i} \left[1-H(t^*_i)\right]^{z_i} \\
&= L(F)L^*(H),
\end{align*}
where $L$ and the $z_i$ are defined as in Equation \ref{eqn:likelihood}. Therefore, the marginal likelihood for $F$ is $L^{\textrm{marginal}}(F) = L(F)E_H[L^*(H)] \propto L(F)$, where the constant of proportionality only depends on the data and $E_H[\cdot]$ represents expectation with respect to the prior distribution of $H$. As a consequence, the posterior distribution of $F$ can be computed ignoring the randomness in the censoring times $C_1,\ldots,C_n$ by considering their observed values as fixed and their unobserved values as fixed to $+\infty$. Hence, if $F$ is a priori a subdistribution beta-Stacy process, then its posterior distribution is the same as in Theorem \ref{thm:conjugacy}. 
\end{proof}

\begin{remark}
The update-rule of Theorem \ref{thm:conjugacy} could be shown to hold under even more general censoring mechanisms. In fact, the marginal likelihood for $F$ remains proportional to $L(F)$ as long as i) the distribution $H$ of censoring times is independent of $F$ and ii) censoring only depends on the past and outside variation \citep{Kalbfleisch2002}.
\end{remark}

\section{Relation with other prior processes}\label{sec:othproc}

\subsection{Relation with the beta-Stacy process}

By construction, the subdistribution beta-Stacy process can be regarded as a direct generalization of the beta-Stacy process. In fact, the two processes are linked with each other, as highlighted by the following theorem:

\begin{theorem}\label{thm:subbetabeta}
A random subdistribution function $F$ is a discrete-time subdistribution beta-Stacy process with parameters $\{(\alpha_{t,0},\ldots,\alpha_{t,k}):t\geq 1\}$ if and only if i) $G(t)=\sum_{d=1}^k F(t,d)$ is a discrete-time beta-Stacy process with parameters $\{(\sum_{d=1}^k \alpha_{t,d},\alpha_{t,0}):t\geq 1\}$ and ii) $\Delta F(t,c) = V_{t,c}\Delta G(t)$ for all $t\geq 1$ and $c=1,\ldots,k$, where $\{V_t=(V_{t,1},\ldots,V_{t,k}): t\geq 1\}$ is a sequence of independent random vectors independent of $G$ and such that $V_t\sim \textrm{Dirichlet}_k(\alpha_{t,1},\ldots,\alpha_{t,k})$ for all $t\geq 1$ (where, if $k=1$, we let the distribution $\textrm{Dirichlet}_1(\alpha_{t,1})$ be the point mass at 1).
\end{theorem}
\begin{proof}
Before proceeding, first observe that $\{(\alpha_{t,0},\ldots,\alpha_{t,k}):t\geq 1\}$ satisfies the recurrency condition of Equation (\ref{eqn:reccondSBS}) if and only if $\{(\sum_{d=1}^k \alpha_{t,d},\alpha_{t,0}):t\geq 1\}$ satisfies the recurrency condition for the beta-Stacy process, i.e. Equation (\ref{eqn:reccondBS}) with $\beta_t=\sum_{d=1}^k \alpha_{t,d}$ and $\gamma_t=\alpha_{t,0}$. 

Now, to prove the ``if'' part of the thesis, suppose that the random subdistribution function $F$ is subdistribution beta-Stacy with parameters $\{(\alpha_{t,0},\ldots,\alpha_{t,k}):t\geq 1\}$. Let $G(t)=\sum_{d=1}^k F(t,d)$ for all integers $t\geq 0$ and define $U_t=\sum_{d=1}^k W_{t,k}$ and $V_{t,c}=W_{t,c}/U_t$ for all $t\geq 1$ and $c=1,\ldots,k$. From these definitions it is easy to check that $\Delta F(t,c) = V_{t,c}\Delta G(t)$ for all $t\geq 1$. Additionally, by standard properties of the Dirichlet distribution \citep[Propriety 1B]{Sivazlian1981} and by the independence of the $W_t$, $\{U_t:t\geq 1\}$ is a sequence of independent random variables such that $U_t\sim Beta(\sum_{d=1}^k \alpha_{t,k},\alpha_{t,0})$. Moreover, from Theorem 2.5 of \citet{Ng2011} it follows that $\{V_{t}=(V_{t,1},\ldots,V_{t,k}):t\geq 1\}$ is a sequence of independent random vectors such that $V_t\sim \textrm{Dirichlet}_k(\alpha_{t,1},\ldots,\alpha_{t,k})$ for all $t\geq 1$, independently of $\{U_t:t\geq 1\}$. Moreover, $G(0)=0$ with probability one because $G(0)=\sum_{d=1}^k F(0,d)$ and $F(0,d)=0$ with probability 1 for all $d=1,\ldots,k$. Continuing, since $\Delta G(t)=\sum_{d=1}^k \Delta F(t,d)$ for all $t\geq 1$, it follows that 
$$\Delta G(t)=\sum_{d=1}^k\left\{ W_{t,d}\prod_{u=1}^{t-1}\left(1-\sum_{c=1}^k W_{u,c}\right)\right\}=U_t\prod_{u=1}^{t-1}(1-U_u)$$ 
for all $t\geq 1$. Thus $G$ is a beta-Stacy process with parameters $\{(\sum_{d=1}^k \alpha_{t,k},\alpha_{t,0}):t\geq 1\}$. 

To prove the ``only if'' part of the thesis, suppose instead that $G$ is a beta-Stacy process with parameters $\{(\sum_{d=1}^k \alpha_{t,d},\alpha_{t,0}):t\geq 1\}$ and that $\{V_t=(V_{t,1},\ldots,V_{t,k}): t\geq 1\}$ is a sequence of independent random vectors satisfying conditions (a) and (b). Since $0=G(0)=\sum_{d=1}^{k}F(0,d)$ with probability 1, and since it must also be $F(t,d)\geq 0$ with probability 1 for all $t$ and $d$, it follows that $F(0,d)=0$ with probability 1 for all $d$. To continue, define $W_t=(W_{t,0},W_{t,1},\ldots,W_{t,k})=(1-U_t,U_t V_{t,1},\ldots,U_t V_{t,k})$ for all $t\geq 1$. It can be seen that the $W_t$ are independent. Moreover, since $U_t$ and $V_t$ are independent and since $1-U_t\sim Beta(\alpha_{t,0},\sum_{d=1}^k \alpha_{t,d})$, from  Theorem 2.2 of \citet{Ng2011}, it follows that $W_t$ has the same distribution as
$$\left(Y_{t,0},Y_{t,1}(1-Y_{t,0}),Y_{t,2}\prod_{d=0}^1(1-Y_{t,d}),\ldots,Y_{t,k-1}\prod_{d=0}^{k-2}(1-Y_{t,d}),\prod_{d=0}^{k-1}(1-Y_{t,d})\right)$$
where the $Y_{t,c}$ are independent random variables with $Y_{t,c}\sim Beta\left(\alpha_{t,c},\sum_{d=c+1}^k \alpha_{t,d}\right)$ for all $c=0,\ldots,k$. Again from Theorem 2.2 of \citet{Ng2011} it thus follows that $W_t\sim \textrm{Dirichlet}_{k+1}(\alpha_{t,0}$, $\ldots$, $\alpha_{t,k})$ for all $t\geq 1$. Since $\Delta F(t,c) = V_{t,c}\Delta G(t)$ for all $t\geq 1$, from the definition of a beta-Stacy process it now follows that 
\begin{align*}
\Delta F(t,d) = V_{t,d} U_{t}\prod_{u=1}^{t-1}\left(1-\sum_{c=1}^k V_{t,c} U_{u}\right) = W_{t,d} \prod_{u=1}^{t-1}\left(1-\sum_{c=1}^k W_{t,c}\right).
\end{align*}
Hence, $F$ is subdistribution beta-Stacy with parameters $\{(\alpha_{t,0},\ldots,\alpha_{t,k}):t\geq 1\}$.
\end{proof}

\subsection{Relation with the beta process}

Suppose $A(t)=(A_1(t),\ldots,A_k(t))$ collects the cumulative hazards of the subdistribution function $F(t,c)$ and let $\Delta A(t)=(\Delta A_1(t),\ldots,\Delta A_k(t))$, $A_0(t)=\sum_{d=1}^k A_d(t)$. Then, following Hjort \cite[Section 2]{Hjort1990}, a discrete time beta-process prior for non-homogeneous Markov Chains with parameters $\{(\alpha_{t,0},\ldots,\alpha_{t,k}):t\geq 1\}$ could be specified for $A(t)$ by independently letting 
$(1- \Delta A_0(t),\Delta A_1(t),\ldots, \Delta A_k(t))$ have a $\textrm{Dirichlet}(\alpha_{t,0},\ldots,\alpha_{t,k})$ distribution for all $t\geq 1$. In such case, from Definition \ref{defin:subbetastacy} it would follow that $F$ is subdistribution beta-Stacy with the same set of parameters. The converse is also true, since if $F$ is subdistribution beta-Stacy then it can be easily seen from Definition \ref{defin:subbetastacy} that $(1-\Delta A_0(t),\Delta A_1(t),\ldots, \Delta A_k(t))=(W_{t,0},W_{t,1},\ldots,W_{t,k})$.
Thus, if interest is in the subdistribution function $F(t,c)$ itself, one should consider the subdistribution beta-Stacy process, whereas if interest is in the cumulative hazards $A(t)$, one should consider the beta process for non-homogeneous Markov Chains. This equivalence parallels an analogous relation between the usual beta-Stacy and beta processes \citep{Walker1997}. 

\subsection{Relation with the beta-Dirichlet process}

The subdistribution beta-Stacy process is also related to the discrete-time version of the \emph{beta-Dirichlet} process, a generalization of Hjort's beta process prior \citep{Hjort1990} introduced by \citet{Kim2012}. The cumulative hazards $\{A(t):t\geq 1\}$ are said to be a beta-Dirichlet process with parameters $\{(\beta_{t,1},\beta_{t,2},\gamma_{t,1},\ldots,\gamma_{t,k}):t\geq 1\}$ if i) the $\Delta A(t)$ are independent, ii) $\Delta A_0(t)\sim Beta(\beta_{t,1},\beta_{t,2})$ for all $t\geq 1$, and iii) $\Delta A(t)/\Delta A_0(t)\sim \textrm{Dirichlet}_k(\gamma_{t,1},\ldots,\gamma_{t,k})$ independently of $\Delta A_0(t)$ for all $t\geq 1$. From Definition \ref{defin:subbetastacy} it is clear that if $F(t,c)$ is subdistribution beta-Stacy with parameters $\{(\alpha_{t,0},\ldots,\alpha_{t,k}):t\geq 1\}$, then from $(1-\Delta A_0(t),\Delta A_1(t),\ldots, \Delta A_k(t))=(W_{t,0},W_{t,1},\ldots,W_{t,k})$ and Theorem 2.5 of \citet{Ng2011}, then the corresponding cumulative hazards $A(t)$ must be beta-Dirichlet with parameters $\beta_{t,1}=\sum_{d=1}^k \alpha_{t,d}$, $\beta_{t,2}=\alpha_{t,0}$, and $\gamma_{t,d}=\alpha_{t,d}$ for all $d=1,\ldots,k$ and $t\geq 1$. The converse is not true unless $\beta_{t,1}=\sum_{d=1}^k \gamma_{t,d}$ for all $t\geq 1$. 

\section{Nonparametric cumulative incidence regression}\label{sec:regr}

In this section, we will illustrate a subdistribution beta-Stacy regression approach for competing risks. We consider data represented by a sample of possibly-right censored discrete survival times and cause-of-failure indicators $(t_1^*,d_1^*)$, $\ldots$, $(t_n^*,d_n^*)$. Each observation $(t_i^*,d_i^*)$ is associated with a known vector $w_i$ of predictors. We assume that, as described in the Introduction, the time axis has been discretized according to some fixed partition $0=\tau_0<\tau_1<\tau_2<\cdots$ representing the measurement scale of event times. Hence, $(t_i^*,d_i^*)=(t,d)$ for some $d=1,\ldots,k$ if an event of type $d$ has been observed in the time interval $(\tau_{t-1},\tau_t]$. Instead, $(t_i^*,d_i^*)=(t,0)$ if no event has been observed during $(\tau_{t-1},\tau_t]$ and censoring took place in the same interval.

Our starting point is the assumption that individual observations are exchangeable within each level of the predictor variables $w_i$. This leads us to consider a hierarchical modelling approach akin to that adopted by \citet{Lindley1972}, \citet{Antoniak1974}, and \citet{Cifarelli1978}. In this approach, the observations $(t_1^*,d_1^*),\ldots,(t_n^*,d_n^*)$ are assumed to be independent, each generated by a corresponding subdistribution function $F(t,c; w_i)$ under some censoring mechanism (as described in Section \ref{sec:post}). Then a joint prior distribution is assigned to all the $F(t,c; w_i)$ for $i=1,\ldots,n$. In the usual parametric approach these would simply be assigned a specific functional form $F_0(t,c\mid \theta; w_i)$ by letting $F(t,c; w_i)=F_0(t,c\mid \theta; w_i)$ for all $i=1,\ldots,n$, then assigning a prior distribution to the common parameter vector $\theta$. Consequently, in the parametric approach the only source of uncertainty is that on the value of $\theta$ and not on the functional form of $F_0(t,c\mid \theta; w_i)$. Thus, to incorporate this uncertainty in the model and gain more flexibility in representing the shape of the $F(t,c; w_i)$, we instead adopt a nonparametric perspective \citep{Mueller2013, Hjort2010, Phadia2015, Ghosal2017}. Specifically, we consider the following modelling approach: first, a specific functional form $F_0(t,c\mid \theta; w_i)$ is chosen; second, letting $w_{(1)},\ldots,w_{(L)}$ denote the distinct values of $w_1,\ldots,w_n$, the subdistribution functions $F(\cdot;w_{(i)})$ are assumed to be independent and distributed as $F(\cdot; w_{(i)})\sim \textrm{SBS}(\omega(\theta,w_{(i)}),F_{0}(\cdot\mid \theta,w_{(i)}))$ for all $i=1,\ldots,L$, where $\omega(\theta,w_{(i)})=(\omega_t(\theta,w_{(i)}))_{t\geq 1}$; finally, the parameter vector $\theta$ is assigned its own prior distribution. The use of weights $\omega_t(\theta,w_{(i)}))$ dependent on the time location $t$, the covariate values $w_{(i)}$, and the parameter $\theta$ allows the local control of the uncertainty on the functional form of $F_{0}(\cdot\mid \theta,w_{(i)})$ by determining the concentration of the subdistribution beta-Stacy prior around the increments $\Delta F_{0}(t,c\mid \theta,w_{(i)})$, $c=1,\ldots,k$ (c.f. Remark \ref{note:priorspecification}). Additionally, the presence of the parameter vector $\theta$ in the model allows borrowing of information across different covariate levels $w_{(i)}$, as in \citet{Cifarelli1978}, \citet{Muliere1993}, and \citet{Mira1996}.  

Many options are available in the literature for specifying the functional form of the centering parametric subdistribution $F_0(t,c\mid \theta,w_{i})$ when adopting our modelling approach. In general, following \citet{Larson1985}, a useful strategy consists in starting from the decomposition
$$F_0(t,c\mid \theta,w_{i})=F_0^{(1)}(c\mid \theta_1,w_i)F_0^{(2)}(t\mid \theta_2,c,w_i),$$
and then separately modelling the probability $F_0^{(1)}(c\mid \theta_1,w_i)$ of observing a failure of type $c$ and the conditional time-to-event distribution $F_0^{(2)}(t\mid \theta_2,c,w_i)$ given the specific failure type $c$. For example, $F_0^{(1)}(c\mid \theta,w_i)$ can be specified to be any model for multinomial responses, such as the familiar multinomial logistic regression model  
\begin{equation}\label{eqn:multlogreg}
F_0^{(1)}(c\mid \theta_1,w_i)=\frac{\exp(w_i'b_c)}{1+\sum_{d=1}^{k-1} \exp(w_i'b_d)},\quad F_0^{(1)}(k\mid \theta_1,w_i)=\frac{1}{1+\sum_{d=1}^{k-1} \exp(w_i'b_d)},
\end{equation}
where $c=1,\ldots,k-1$ and $\theta_1=(b_1,\ldots,b_{c-1})$ \citep[Chapter 7]{Agresti2003}. The time-to-event distribution $F_0^{(2)}(t\mid \theta_2,c,w_i)$ can instead be specified by discretizing a continuous-time distribution (e.g. Weibull or log-normal) with cumulative distribution function $G_0(\cdot\mid \theta_2,c,w)$ by letting $F_0^{(2)}(t\mid\theta_2,c,w_i) = G_0(\tau_t\mid \theta_2,c,w_i)$. For example, in the Weibull case one could let 
\begin{equation}\label{eqn:weibullreg}
G_0(t\mid \theta_2,c,w_i)=1-\exp(-t^{u_c}\exp(w_i'v_c))
\end{equation}
and $\theta_2=(v_1,\ldots,v_k,u_1,\ldots,u_k)$, yielding a discrete-time version of the common parametric Weibull regression model \citep[Chapter 5]{Aalen2008}. Alternatively, $F_0^{(2)}(t\mid \theta_2,c,w_i)$ may be specified as the Grouped Cox model \citep[Section 2.4.2]{Kalbfleisch2002}, the logistic regression model of \citet{Cox1972}, or other discrete-time models \citep{Schmid2016,Tutz2016,Berger2017}. 

The weights $\omega_t(\theta,w_{(i)})$ can be calibrated in order to account for the degree of uncertainty attached to the chosen parametric functional form of $F_0(\cdot\mid \theta,w_i)$. In fact,   
by Remark \ref{note:priorspecification}, conditionally on $\theta$ as the weights increase the prior $\textrm{SBS}(\omega,F_{0}(\cdot\mid\theta,w_{(i)}))$ becomes more concentrated on $F_{0}(\cdot\mid\theta,w_{(i)}))$, giving more importance to the parametric component of the model. To exploit this fact, we consider weights $\omega_t$ defined as $\omega_t(\theta,w_{(i)})=\omega_{0,t}(\theta,w_{(i)})/m$, where
$$\omega_{0,t}(\theta,w_{(i)})=\frac{\tau_t-\tau_{t-1}}{\sum_{d=1}^k F_0(\tau_t,d\mid \theta,w_{(i)})-\sum_{d=1}^k F_0(\tau_{t-1},d\mid        \theta,w_{(i)})}.$$
Extending the approach of \citet{Rigat2012}, this choice allows the model to rely more on its parametric component over the times where observations are less likely to be available (high $\omega_{0,t}$), whereas it allows for more flexibility over the times where most data is expected (low $\omega_{0,t}$). The parameter $m$ can be further used to control the standard deviations 
\begin{equation*}\label{eqn:stddev}
\sigma_m(t,c;w_{(i)})=\sqrt{\Var{\Delta F(t,c\mid w_{(i)})-\Delta F_0(t,c\mid \theta,w_{(i)})}}
\end{equation*}
(computed from the joint distribution of $F(t,c\mid w_{(i)})$ and $\theta$), which together measure how much the subdistribution function $F(\cdot; w_{(i)})$ can deviate from the parametric model $F_0(\cdot\mid\theta,w_{(i)})$. More precisely, by Remark \ref{note:priorspecification}, for all fixed $t\geq 1$ and $c=1,\ldots,k$, is a decreasing function of $m$. In particular, $\sigma_m(t,c;w_{(i)})\rightarrow 0$ for all $t$ and $c$ as $m\rightarrow 0$, implying that for $m\approx 0$ the model becomes essentially equivalent to the centering parametric model. Conversely, $\sigma_m(t,c;w_{(i)})$ increases to its maximum value 
\begin{equation}\label{eqn:maxstddev}
\sigma_\infty(t,c;w_{(i)})=\sqrt{E[\Delta F_0(t,c\mid \theta,w_{(i)})(1-\Delta F_0(t,c\mid \theta,w_{(i)}))]}
\end{equation}
for all $t$ and $c$ as $m\rightarrow+\infty$. Hence, for large $m$  the model becomes more flexible and is allowed to deviate more freely from the centering parametric model.    

\begin{remark}
Conditionally on $\theta$, the predictive structure of such model can be characterized as by associating an urn system like that described in Section \ref{sec:predconstr} to each distinct value of $w_{(i)}$. The initial composition of these urns is determined by $\alpha_{t,c}(\theta,w_{(i)}) = \omega_{0,t}(\theta,w_{(i)}) \Delta F_0(t,c\mid \theta,w_{(i)})$ and $\alpha_{t,0}(\theta,w_{(i)}) = \omega_{0,t}(\theta,w_{(i)}) \left(1-\sum_{d=1}^k F_0(t,d\mid \theta,w_{(i)})\right)$. If each extracted ball is reinforced by $m$ similar balls, then by Theorem \ref{thm:reprthm} and Remark \ref{remark:m}, the distributions associated to the same value of $w_{(i)}$ are independent and each distributed according to some $F(\cdot\mid w_{(i)})\sim \textrm{SBS}((\omega_t(\theta,w_{(i)}))_{t\geq 1},F_{0}(\cdot\mid \theta,w_{(i)}))$, where $\omega_t(\theta,w_{(i)})=\omega_{0,t}(\theta,w_{(i)})/m$ as above. 
\end{remark}

\subsection{Sampling from the posterior distribution}

To fix notations, let ${t}^*=(t_1^*,\ldots,t_n^*)$, ${d}^*=(d_1^*,\ldots,d_n^*)$, ${w}=(w_1,\ldots,w_n)$, and $\mcal{F}=(F(\cdot; w_{(1)})$, $\ldots$, $F(\cdot;w_{(L)}))$. Also, let $n_j=\sum_{i=1}^n\Ind{w_i=w_{(j)}}$ and for all $i=1,\ldots,n_j$ let $(t^*_{j,1},d^*_{j,1})$, $\ldots$, $(t^*_{j,n_j},d^*_{j,n_j})$ be the set of observations corresponding to the value $w_{(j)}$. Lastly, let $z_{j,i}=\Ind{d_{j,i}\neq 0}$ for all possible $j$ and $i$. Finally, let $t_j^*=(t^*_{j,i}:i=1,\ldots,n_j)$ and $d_j^*=(d^*_{j,i}:i=1,\ldots,n_j)$ for all $j=1,\ldots,L$.

\begin{theorem}
Assuming ignorable right censoring, the marginal likelihood of $\theta$ is 
$$P({t}^*,{d}^*\mid \theta,{w}) = \prod_{j=1}^L\prod_{i=1}^{n_j} \left\{\Delta F^*_{j,i-1}(t_{j,i}^*,d_{j,i}^*\mid\theta,w_{(j)})^{z_{j,i}} \left[1-\sum_{d=1}^k F_{j,i-1}^*(t^*_{j,i},d\mid\theta,w_{(j)})\right]^{1-z_{j,i}}\right\},$$
where, for all $j=1,\ldots,m$: i) for all $i=1,\ldots,n_j$, $F^*_{j,i}(t,d\mid \theta,w_{(j)})$ is the predictive distribution of a new observation from $F(\cdot\mid w_{(j)})$ given $(t^*_{j,1},d^*_{j,1})$, $\ldots$, $(t^*_{j,i},d^*_{j,i})$, obtained from Corollaries \ref{note:preddist} and \ref{note:postsbs}; ii) $F^*_{j,0}(t,d\mid \theta,w_{(j)})=F_0(t,d\mid\theta,w_{(j)})$.
\end{theorem}
\begin{proof}
First assume that censoring is fixed. In this case, the marginal likelihood of $\theta$ can be obtained from conditional likelihood
\begin{align*}
P(t^*,d^*\mid \theta,w,\mcal{F}) &= \prod_{j=1}^L P(t^*_j,d^*_j\mid\theta,w_{(j)},F(\cdot;w_{(j)})) \\
&= \prod_{j=1}^L P(t^*_j,d^*_j\mid F(\cdot;w_{(j)}))
\end{align*}
by taking its expectation with respect to the distribution of $\mcal{F}$ conditional on $\theta$. Since the $F(\cdot;w_{(j)})$ are independent conditionally on $\theta$, the marginal likelihood is thus
\begin{align*}
P(t^*,d^*\mid \theta,w) &= \prod_{j=1}^L P(t^*_j,d^*_j\mid\theta,w_{(j)}) \\
&= \prod_{j=1}^L \prod_{i=1}^{n_j} P(T^*_{j,i}=t^*_{j,i},D^*_{j,i}=d^*_{j,i}\mid t^*_{j,h}, d^*_{j,h}, h<i; \theta,w_{(j)}), 
\end{align*}
where $P(T^*_{j,i}=t^*_{j,i},D^*_{j,i}=d^*_{j,i}\mid t^*_{j,h}, d^*_{j,h}, h<i; \theta,w_{(j)})$ is the conditional predictive distribution of $(t^*_{j,i},d^*_{j,i})$ given all $(t^*_{j,h}, d^*_{j,h})$ with $h<i$ and $\theta$. If $z^*_{j,i}=1$, this can be derived from Corollaries \ref{note:postsbs} and \ref{note:preddist} and it is equal to $\Delta F^*_{j,i-1}(t^*_{j,i},d^*_{j,i}\mid \theta)$. If instead $z^*_{j,i}=0$, then this is equal to 
\begin{align*}
P(T_{j,i}>t^*_{j,i}\mid (T^*_{j,h}, D^*_{j,h})=(t^*_{j,h}, d^*_{j,h}), h<i; \theta,w_{(j)})=1-\sum_{d=1}^k F^*_{j,i-1}(t^*_{j,i},d\mid\theta).
\end{align*}
This justifies the thesis if censoring is fixed. By similar arguments as those in Section \ref{sec:post}, the same likelihood can be assumed to hold also in presence of ignorable censoring, as needed.
\end{proof}

Using the above result, the joint posterior distribution $P(\mcal{F},\theta\mid t^*,d^*,w)$ of $\mcal{F}$ and $\theta$ can be obtained as 
\begin{equation}\label{eqn:post}
P(\mcal{F},\theta\mid t^*,d^*,w)\propto P(\theta)P({t}^*,{d}^*\mid  \theta,{w})\prod_{j=1}^L P_j(F(\cdot;w_{(j)})\mid \theta,w),
\end{equation}
where $P(\theta)$ represents the prior distribution of $\theta$ (which is independent of $w$) and the term $P_j(F(\cdot;w_{(j)})\mid \theta,w)$ represents the posterior distribution of $F(\cdot; w_{(j)})\sim \textrm{SBS}(\omega,F_{0}(\cdot\mid \theta,w_{(j)}))$ obtained (for fixed $\theta$) from the data $\mcal{D}_j=\{(t_i^*,d_i^*):w_i=w_{(j)}, i=1,\ldots,n\}$ using the update rule described in Theorem \ref{thm:conjugacy}. Now, although the posterior distribution for $\theta$ is not available for exact sampling, Equation (\ref{eqn:post}) suggests the use of a Markov Chain Monte Carlo strategy such as the following to perform approximate posterior inferences. First, a sample $\{\theta_i\}_{i=1}^S$ from the marginal posterior distribution of $\theta$ is obtained, after discarding an appropriate number of burn-in iterations, via a Random Walk Metropolis-Hastings algorithm \citep[Section 7.5]{Robert2004}. A multivariate Gaussian distribution can be considered after the reparametrization induced by a logarithmic transformation of each shape parameter $u_c$ (to account for their positive support).  Second, having obtained a sample $\{\theta_i\}_{i=1}^S$ as just described, the conditional posterior distribution of $F(\cdot;w_{(j)})$ given $\theta_i$ and the data $\mcal{D}_j$ is obtained by direct simulation for all $i=1,\ldots,S$ and $j=1,\ldots,L$. Specifically, the parameters of the conditional posterior distribution
$P_j(F(\cdot;w_{(j)})\mid \theta,w)$ of $F(\cdot;w_{(j)})$ given $\theta_i$ and $\mcal{D}_j$ are obtained using Theorem \ref{thm:conjugacy}. Then a sample $F_i(\cdot;w_{(j)})$ from $P_j(F(\cdot;w_{(j)})\mid \theta,w)$ is obtained using Definition \ref{defin:subbetastacy} by sampling from the relevant Dirichlet distributions. The sample $\{(\theta_i,F_i(\cdot;w_{(1)}),\ldots,F_i(\cdot;w_{(L)}))\}_{i=1}^S$ so obtained then represents a sample from the joint posterior distribution of Equation (\ref{eqn:post}).

\subsection{Estimating the predictive distributions}

Let $T_{n+1}$ and $D_{n+1}$ be the unknown uncensored survival time and type of realized outcome, respectively, for a new individual with covariate profile $w_{n+1}$. The objective is to estimate the predictive distribution of $(T_{n+1},D_{n+1})$ given the data $(t_1^*,d_1^*)$, $\ldots$, $(t_n^*,d_n^*)$. We distinguish two cases: i) $w_{n+1}=w_{(j)}$ for some $j=1,\ldots,L$, and ii) $w_{n+1}\neq w_{(1)},\ldots,w_{(L)}$. In the first case, simply obtain a sample $\{F_i(\cdot;w_{n+1})=F_i(\cdot;w_{(j)})\}_{i=1}^S$ from the posterior distribution of $F(\cdot;w_{(j)})$ using the output of the procedure described above. The predictive distribution of $(T_{n+1},D_{n+1})$ is then estimated as $S^{-1}\sum_{i=1}^S F_i(\cdot;w_{n+1})$. In the second case it is still possible to estimate the predictive distribution of $(T_{n+1},D_{n+1})$ by recycling the sample $\{\theta_i\}_{i=1}^S$. Specifically, for each $\theta_i$, $F_i(\cdot;w_{n+1})$ is simulated directly from the $\textrm{SBS}(\omega,F_0(\cdot\mid \theta_i,w_{n+1}))$ distribution. The predictive distribution of $(T_{n+1},D_{n+1})$ is then estimated as the average of the sampled subdistribution functions, as before.

\section{Application: analysis of the melanoma dataset}\label{sec:application}

\subsection{Data description and analysis objectives}
\label{sec:data}

To illustrate our modelling approach, we analyse data collected by \citet{Drzewiecki1980} on 205 stage I melanoma patients who underwent surgical excision of the tumor during 1962-1977 at the Odense University Hospital, Denmark. This dataset (which includes only data for those 205 patients for which an histological examination was carried out, out of the 225 originally participating in the study) has been previously used to illustrate several survival analysis methods \citep[Example I.3.1]{Andersen2012} and is freely available online as part of the \emph{timereg} R library \citep{Scheike2011}. Each considered patient was followed from the date of surgery to the time of death for melanoma (event of type 1), death due to other causes (event of type 2), or censoring (e.g. study drop-out or end of the study, defined at the end of 1977). Event times are only known discretized at the day level, so that $\tau_t=t$ for all $t\geq 0$ can be assumed. (The time-interval $(\tau_{t-1},\tau_t]$ represents the $t$-th day of follow-up). In summary, 126 ($61\%$) of the study participants were women and 79 ($39\%$) were men. Overall, a total of 57 ($28\%$) patients died due to melanoma during follow-up, while 14 ($7\%$) died due to other causes, overall accumulating 441,324 person-days of follow-up (maximum follow-up: men, 4,492 days; women, 5,565 days). Using these data, we implement a competing-risks regression model to assess the long-term prognosis of melanoma patients following surgical excision of the tumor with respect to the risk of death due to melanoma. In doing so, we account for death due to other causes as a competing event and consider gender as a potential predictor. The R code used to perform the analyses is available as on-line Supplementary Material and at \url{https://github.com/andreaarfe/subdistribution-beta-stacy}.

\subsection{Model specification and prior distributions}\label{sec:applimodel}

We consider a regression model specified as explained in Section \ref{sec:regr}. For illustration, we specify the centering parametric model $F_0(t,c\mid\theta,w_{i})$ by consider the multinomial logistic model (\ref{eqn:multlogreg}) for $F_0^{(1)}(c\mid \theta_1,w_i)$ and the discrete Weibull regression model (\ref{eqn:weibullreg}) for $F_0^{(2)}(t\mid \theta_2,c,w_i)$, as these correspond to models widely used in applications. In these models, for all subjects $w_i=(w_{i,1},w_{i,2})$ includes an intercept term ($w_{i,1}=1$) and the indicator variable for gender ($w_{i,2}=0$ for women, $w_{i,2}=1$ for men). Consequently, in the notations of Section \ref{sec:regr}, $\theta_1=(b_1)$, where $b_1=(b_{1,1},b_{1,2})$ is the vector of the two regression coefficients in model (\ref{eqn:multlogreg}) ($b_{1,1}$ for the intercept, $b_{1,2}$ for the gender indicator). Additionally, $\theta_2=(v_1,v_2,u_1,u_2)$, where $v_c=(v_{c,1},v_{c,2})$ is the vector of the two regression coefficients for the cause-specific Weibull regression model (\ref{eqn:weibullreg}) for $c=1,2$ ($v_{c,1}$ for the intercept, $v_{c,2}$ for the gender indicator), while $u_1,u_2>0$ are the two corresponding shape parameters. We assign independent prior distributions to all parameter as follows. Noting that \citet{Drzewiecki1980} estimated that the overall 10-years survival probability was about 50\% (estimated via the Kaplan-Meier method) in a previous analysis of a larger dataset, we calibrate the priors for $v_1$ and $v_2$ in such a way so as to center the curves $F_0^{(2)}(t\mid \theta_2,c,w_i)$ around a model with a median survival of 3,650 days. To do so, we assigned $N(\log(-\log(0.50)/3,650),1)$ priors to $v_{1,1}$ and $v_{2,1}$, and $N(0,1)$ to $v_{1,2}$ and $v_{2,2}$. We assign a $N(0,1)$ prior distributions to $b_{1,1}$, $b_{1,2}$ and a gamma distribution $Gamma(g_1,g_2)$ with shape parameter $g_1=11$ and rate parameter $g_2=10$ distribution to $u_1$ and $u_2$ (thus centering the corresponding Weibull distributions on an exponential model). Numerical simulations reported in the on-line Supplementary Material (Appendix A, Figure A.1) suggest that these choices yield a fairly diffuse prior distribution for the subdistribution function of the model. As a sensitivity analysis, in the online supplementary Material (Appendix C), we report the results obtained from a similar model but considering a discrete log-normal distribution for the centering parametric subdistribution function.

\subsection{Calibrating the prior concentrations}\label{sec:priorcalib}

To illustrate the behaviour of our model as $m$ varies, Figure \ref{fig:prior_stddev} shows, for the prior distributions specified in the previous section, the values of the prior standard deviations $\sigma_m(t,c;w_{(i)})$ for increasing values of $m$, computed by simulating from the priors described in the previous setting and focusing on the subdistribution of death due to melanoma ($c=1$) among women ($w_{(i)}=(1,0)$). Qualitatively identical results (data not shown) can be obtained for death due to other causes ($c=2$) or men ($w_{(i)}=(1,1)$). These results show how, for small $m$ (e.g. $m=1$ in Figure \ref{fig:prior_stddev}), the nonparametric prior $\textrm{SBS}(\omega(\theta,w_{(i)}),F_{0}(\cdot\mid \theta,w_{(i)}))$ for $F(t,c;w_{(i)})$ is practically fully concentrated on its parametric component (as $\sigma_m(t,c;w_{(i)})\approx 0$ for most $t\geq 1$). This implies that for small $m$ the subdistribution $F(t,c;w_{(i)})$ will tend to be almost equal to the parametric centering subdistribution $F_0(t,c\mid \theta; w_{(i)})$ a priori. However, as $m$ increases, so does $\sigma_m(t,c;w_{(i)})$, representing increasing levels of prior uncertainty on the functional form of $F(t,c;w_{(i)})$. For sufficiently large values of $m$ (e.g. $m=10^5$ in Figure \ref{fig:prior_stddev}), the $\sigma_m(t,c;w_{(i)})$ achieve values close to their upper bound $\sigma_\infty(t,c;w_{(i)})$, which represents a situation of maximum uncertainty on the functional form of $F(t,c;w_{(i)})$. Regardless of the value of $m$, the $\sigma_m(t,c;w_{(i)})$ decrease as $t\geq 1$ increases, showing how the parametric model $F_0(t,c\mid \theta; w_{(i)})$ is given more and more weight in determining the form of $F(t,c;w_{(i)})$ over the later portions of follow-up (i.e. over times where less data is expected a priori). These observations both illustrate the consideration of Section \ref{sec:regr} but also suggest that plots like Figure \ref{fig:prior_stddev} may be useful in practice to calibrate the prior concentration.

\begin{figure}
\centering
\includegraphics[scale=0.7]{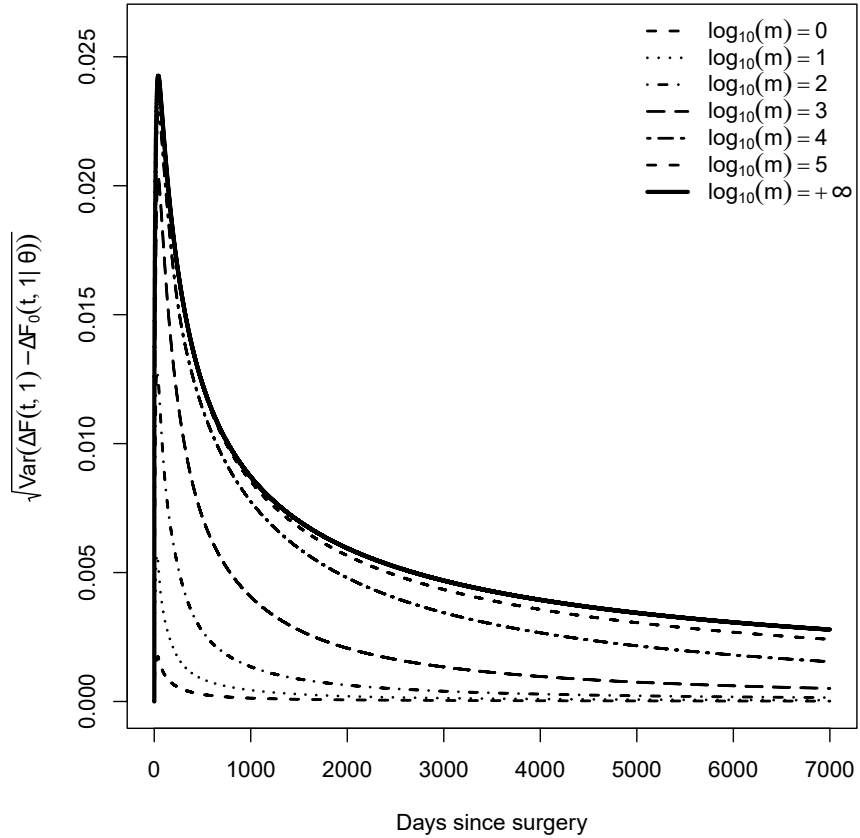}
\caption{Prior standard deviations $\sigma_m(t,c;w_{(i)})=\sqrt{\Var{\Delta F(t,c\mid w_{(i)})-\Delta F_0(t,c\mid \theta,w_{(i)})}}$ corresponding to the model of Section \ref{sec:applimodel} for $m=10^0,10^1,\ldots,10^5$, together with its upper bound $\sigma_\infty(t,c;w_{(i)})$ from Equation \ref{eqn:maxstddev}. Results are for the subdistribution of death due to melanoma ($c=1$) among women ($w_{(i)}=(1,0)$). The quantity $\sigma_m(t,c;w_{(i)})$ measures the concentration of the subdistribution beta-Stacy prior for the subdistribution function $F(t,c;w_{(i)})$ around the parametric subdistribution $F_0(t,c\mid  \theta,w_{(i)})$. Values of $\sigma_m(t,c;w_{(i)})\approx 0$ signify that $F(t,c;w_{(i)})\approx F_0(t,c\mid \theta,w_{(i)})$ with high priori probability, while increasing values of $\sigma_m(t,c;w_{(i)})>0$ signify that functional forms different than $F_0(t,c\mid \theta,w_{(i)})$ are more likely a priori.}
\label{fig:prior_stddev}
\end{figure}

\subsection{Posterior analysis}

\begin{figure}
\captionsetup[subfigure]{justification=centering}

\begin{subfigure}{\columnwidth}
  \centering
  \includegraphics[scale=.8]{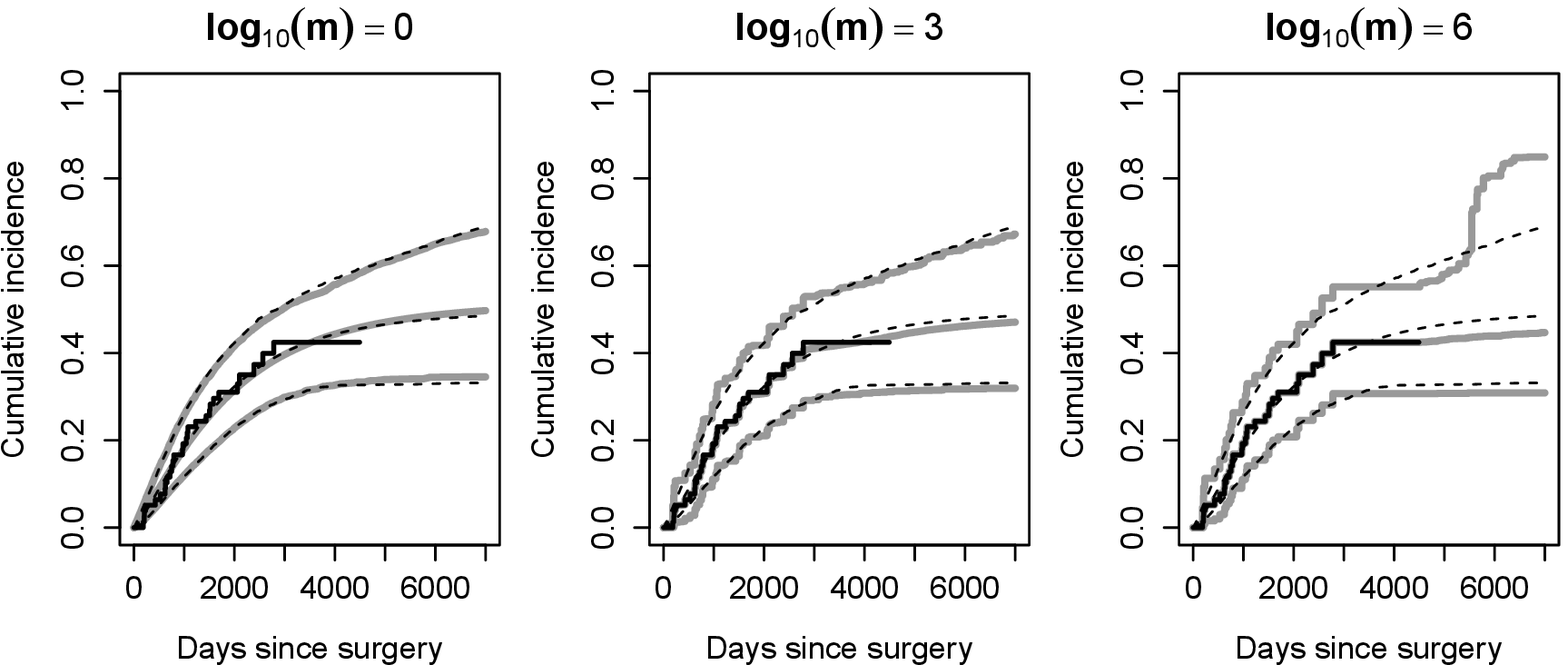}
  \caption{Cumulative incidence of death due to melanoma, men.}
  \label{fig:sfig1}
\end{subfigure}

\medskip

\begin{subfigure}{\columnwidth}
  \centering
  \includegraphics[scale=.8]{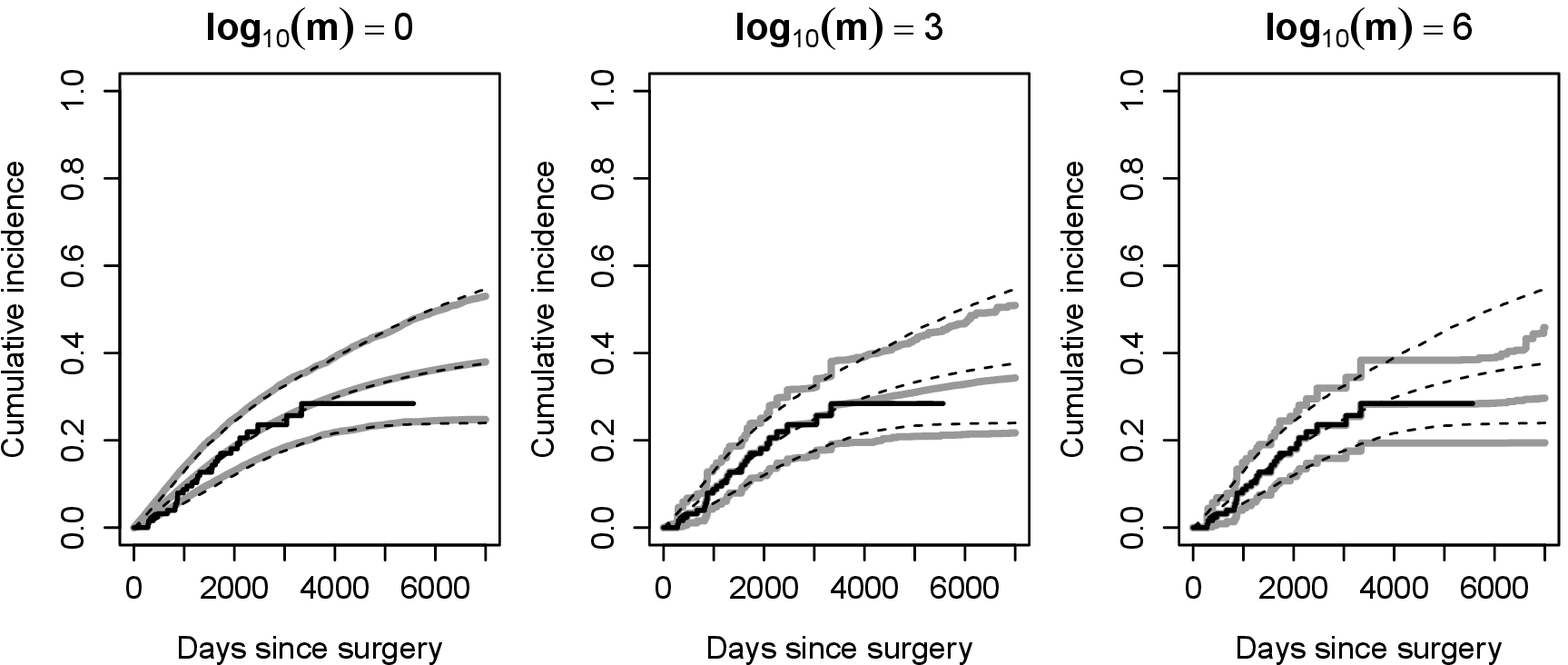}
  \caption{Cumulative incidence of death due to melanoma, women.}
  \label{fig:sfig2}
\end{subfigure}
\caption{Posterior summaries for the cumulative incidence of death due to melanoma among (a) men and (b) women, i.e. posterior summaries for the subdistribution functions $F(t,1;w_{(i)})$ for the model of Section \ref{sec:application} with (a) $w_{(i)}=(1,1)$ and (b) $w_{(i)}=(1,0)$, computed for reinforcement parameters $m=10^0$, $10^3$, and $10^6$. Solid black lines: Kalbfleish-Prentice classical estimators. Solid gray lines: posterior means of the subdistribution function, i.e. posterior predictive distributions, with upper and lower 95\% pointwise credibility limits. Dashed black lines, posterior means and 95\% pointwise credibility limits for the multinomial-Weibull model of Section \ref{sec:regr}. }
\label{fig:preddist}
\end{figure}

Posterior inference was performed by a Random Walk Metropolis-Hastings algorithm with a multivariate Gaussian proposal distribution as suggested in Section \ref{sec:regr}, by means of the \textit{MCMCpack} R package \citep{Martin2011}. The proposal distribution was centered at the current sampled value, with a proposal covariance matrix equal to the negative inverse Hessian matrix of the log-posterior distribution, evaluated at the posterior mode and scaled by $(2.4)^2/d$, where $d$ is the dimension of $\theta$, as suggested by \citet[Section 12.2]{Gelman2013}. To improve mixing, all predictors were standardized before running the algorithm. The parameter vector was initialized with the corresponding value obtained by numerically maximizing the log-posterior distribution. In all cases, the Metropolis-Hastings algorithm was run for a total of 26000 iterations: the first 1000 were discarded as burn-in, while the remaining 25000 were thinned by retaining only one generated sample every 25 iterations. The trace plots of generated Markov Chain Monte Carlo chains did not raise any issue of non-convergence according to both Geweke's test \citep{Geweke1992} and visual inspection (data not shown). Additionally, the obtained posterior distributions were found to be much more concentrated than the considered prior distributions, as shown in the on-line Supplementary Material (Appendix A, Figure A.1). 

Figure \ref{fig:preddist} shows the posterior predictive distributions, i.e. the posterior expectations of the subdistribution functions $F(t,1;w_{(i)})$, for death due to melanoma among men (panel a) and women (panel b), for $m=10^0$, $10^3$, or $10^6$. For comparison, Figure \ref{fig:preddist} also reports i) the estimates obtained from the classical Kalbfleish-Prentice estimator and ii) the posterior estimates obtained from the centering multinomial-Weibull parametric model $F_0(t,c\mid \theta,w_{i})$ of Section \ref{sec:applimodel} (using the same parametric prior distributions for comparability). In general, the results are compatible with the observation that men are subject to a higher risk of death due to melanoma than women \citep{Thoern1994}. Additionally, from Figure \ref{fig:preddist} it is apparent how the classical estimators have a limited usefulness for evaluating long-term prognosis, as these are undefined beyond the range of the observed data. On the other hand, by relying more on its parametric component, our subdistribution beta-Stacy model can provide an extrapolated risk estimate. Risk extrapolations could also be obtained from the centering parametric model, but these would require absolute confidence in its assumed functional form. Oppositely, our model may deviate more or less flexibly from the centering parametric model according to the chosen value of $m$. In fact, the results obtained from the subdistribution beta-Stacy model for $m=1$ are essentially equivalent to those obtained from centering parametric model. However, as $m$ increases the subdistribution beta-Stacy predictive distributions better approximate the classical estimators of the subdistribution function. For $m=10^6$ it can also be seen that the posterior variance of the subdistribution function may be very large beyond the range of the observed data, as seen in Figure \ref{fig:preddist} for men (panel a), i.e. the group that required the most extrapolation for computing the predictive distribution over the considered time period. This behaviour is consistent with the observations of Section \ref{sec:priorcalib}: for large $m$ the model allows more uncertainty on the functional form of the centering model. 

Additional results are provided in the on-line Supplementary material, Appendices B and C. Specifically, in Appendix B we report the results of graphical posterior predictive checks for the goodness of fit of our model, in the style of \citet[Section 6.3]{Gelman2013}. These checks do not raise any concern regarding the fit of our model. In Appendix C, we report the results obtained in the sensitivity analysis based on the discrete log-normal model. The corresponding results are essentially equivalent to the ones obtained here.

\subsection{Simulation study}
\label{sec:sim}

To further explore how much our model can adapt to deviations from the corresponding centering parametric functional form, we conducted a simulation study as follows. First, on the basis of the melanoma data of Section \ref{sec:data}, we computed the maximum likelihood estimates $\widehat{\theta}=(\widehat{b}_{1,1},\widehat{v}_{1,1},\widehat{v}_{2,1},\widehat{u}_1,\widehat{u}_2)$ for the multinomial-Weibull model $F_0(t,c\mid \theta)$ of Section \ref{sec:applimodel}, ignoring covariates for simplicity by including only an intercept term as predictor ($w_{(i)}\equiv 1$). We thus obtained $\widehat{b}_{1,1}=-0.640$, $\widehat{v}_{1,1}=-11.927$, $\widehat{v}_{2,1}=-7.244$, $\widehat{u}_1=1.597$, and $\widehat{u}_2=0.639$. Second, we generated 100 datasets of sample size $n=100$, $n=500$, and $n=1000$ by simulating event times and event types from the subdistribution function $F_0(t,c\mid \widehat{\theta})$, with a fixed censoring time at 7000 days since surgery. Third, in each simulated dataset, we implemented two Bayesian models: i) a multinomial-Weibull parametric model akin to that in Section \ref{sec:applimodel} but where all Weibull shape parameters where fixed as $u_1,u_2\equiv 1$ (all other prior distributions taken as in Section \ref{sec:applimodel}); this corresponds to incorrectly modelling the event times as exponentially distributed given the type of occurring event, incompatibly with the data-generating mechanism; ii) a nonparametric subdistribution beta-Stacy model centered on the parametric model of point (i) for all values of $m=10^0,10^3,10^6$. Fourth and last, for each replicated dataset we computed the Kolmogorov-Smirnov distance $\max_{t\in[0,7000]}\left|\widehat{F}(t,c)-F_0(t,c\mid \widehat{\theta})\right|$ between the fixed data-generating subdistribution function $F_0(t,c\mid \widehat{\theta})$ and its posterior estimate $\widehat{F}(t,c)$, obtained from model (i) or (ii). 

\begin{figure}
\centering
\includegraphics[scale=0.65]{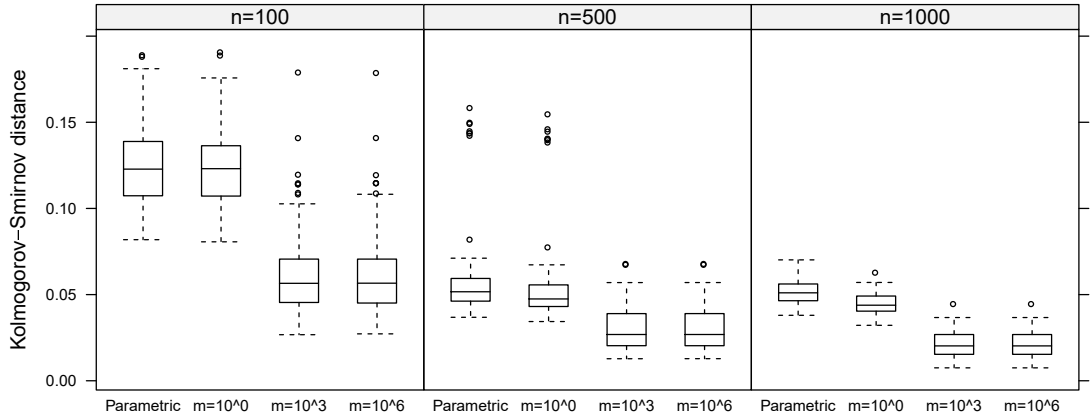}
\caption{Box-plots reporting the distributions of Kolmogorov-Smirnov distances between the data-generating subdistribution function and its posterior estimates generated in the simulation study of Section \ref{sec:sim} for sample sizes $n=100$, $n=500$, or $n=1000$ (considering 100 simulated datasets per sample size). Kolmogorov-Smirnov distances are shown separately for the parametric model described in Section \ref{sec:sim} and the nonparametric subdistribution beta-Stacy model, centered on the same parametric model, for reinforcement parameters of $m=10^0, 10^3, 10^6$.}
\label{fig:simstudy}
\end{figure}

Figure \ref{fig:simstudy} reports the distribution of Kolmogorov-Smirnov distances obtained in the described simulation study. As expected, for all sample size the misspecified parametric model produces the highest median Kolmogorov-Smirnov distances between the posterior estimates of the subdistribution function and the true data-generating subdistribution function. In agreement with previous observations, for $m=1$ the nonparametric subdistribution beta-Stacy model tends to agree with its parametric component, producing similar results as those obtained from the parametric model. For increasing $m$, however, the subdistribution beta-Stacy model attains a greater flexibility to deviate from its misspecified parametric centering model and adapt to the data, thus producing lower Kolmogorov-Smirnov distances. This phenomenon, which is consistent with the observations of Section \ref{sec:priorcalib}, is evident for all considered sample sizes, but especially for $n=1000$. For this sample size, even the subdistribution beta-Stacy model with $m=1$ is associated with a lower median Kolmogorov-Smirnov distance from the data-generating model than its centering parametric model: despite a large prior weight was assigned to the centering mispecified model, the sample size was large enough for the subdistribution beta-Stacy model to be more driven by its nonparametric component and adapt to the data.

\section{Concluding remarks}\label{sec:concl}

In this paper we introduced a novel stochastic process, the subdistribution beta-Stacy process, useful for the Bayesian nonparametric regression analysis of competing risks data. We showed how the subdistribution beta-Stacy process is completely characterized from a specific predictive structure, which we described in terms of the urn-based reinforced stochastic process of \citet{Muliere2000}. The practical value of similar reinforced stochastic processes is that they potentially allow to undertake Bayesian predictive inference without explicit knowledge of the prior. That is, as noted by \citet{Muliere2003}, they allow to update the predictive distributions from past information from a sequence of exchangeable observable without necessarily being able to compute the underlying de Finetti measure of the observations, i.e. the prior. In this paper, we were actually able to characterize the prior, which we identified as the subdistribution beta-Stacy. Although this process may have been defined solely in terms of a sequence of independent Dirichlet random vectors (as in our Definition 2.2), the corresponding predictive construction still greatly simplifies the understanding of its properties.

In this paper we also proposed a Bayesian nonparametric approach for competing risks regression based on the subdistribution beta-Stacy process. This provides several advantages with respect to other available techniques when making predictions in presence of competing risks. For instance, classical nonparametric estimators are typically undefined beyond the last observation time if this is censored, thus limiting their usefulness when making predictions. Parametric models circumvent this issue by assuming a specific functional form for the subdistribution function, thus providing risk extrapolations at the cost of more rigidity when adapting to the data. Conversely, by balancing both a nonparametric and a parametric component, our approach allows risk extrapolations beyond the range of the observed data without losing flexibility in adapting to available data. Additionally, contrary to most approaches available in the literature, our does not require the proportional hazards assumption, thus increasing its flexibility in capturing complex patterns of subdistribution functions when making predictions.  

To conclude, we remark that more general reinforced urn processes may lead to interesting novel approaches for performing Bayesian inference in presence of competing risks or more general settings. In fact, we are currently investigating the following possible generalizations. First, akin as in \citet{Muliere2006}, each extracted ball may be reinforced by a positive random number of new similar balls. This may depend on both the color of the ball extracted from the urn and the state represented by the urn itself. Such construction could be useful to represent allow uncertainty on the strength of belief to be granted to the initial composition of the urns. Second, a continuous-time version of the subdistribution beta-Stacy process could be obtained by embedding a discrete-time reinforced urn process like that of Section \ref{sec:predconstr} into a reinforced continuous-time arrival process (representing the predictive distribution of the event times) as in the approach of \citet{Muliere2003}. Third, the reinforced urn process considered in Section \ref{sec:predconstr} could be generalized to characterize a process prior on the space of transition kernels of a Markovian multistate process, with application to the analysis of event-history data \citep{Aalen2008}. The conceptual difficulty here is that such processes may not be recurrent, complicating the use of representation theorems like those of \citet{Diaconis1980}.

\subsection*{Acknowledgements}
The authors wish to thank the Associate Editor and two anonymous Referees for their useful comments.

\subsection*{Supporting information} Supplementary information is available online at the publisher's web-site. R code to reproduce the analyses is also available at \url{https://github.com/andreaarfe/subdistribution-beta-stacy}.

\bibliography{bibliography}

\begin{thebibliography}{78}
\expandafter\ifx\csname natexlab\endcsname\relax\def\natexlab#1{#1}\fi
\expandafter\ifx\csname url\endcsname\relax
  \def\url#1{\texttt{#1}}\fi
\expandafter\ifx\csname urlprefix\endcsname\relax\def\urlprefix{URL: }\fi

\bibitem[{Aalen et~al.(2008)Aalen, Borgan and Gjessing}]{Aalen2008}
Aalen, O., Borgan, O. and Gjessing, H. (2008) \textit{Survival and event
  history analysis: a process point of view}.
\newblock New York: Springer Science \& Business Media.

\bibitem[{Agresti(2003)}]{Agresti2003}
Agresti, A. (2003) \textit{Categorical Data Analysis}.
\newblock Wiley Series in Probability and Statistics. Wiley.

\bibitem[{Allison(1982)}]{Allison1982}
Allison, P.~D. (1982) Discrete-time methods for the analysis of event
  histories.
\newblock \textit{Sociological methodology}, \textbf{13}, 61--98.

\bibitem[{Amerio et~al.(2004)Amerio, Muliere and Secchi}]{Amerio2004}
Amerio, E., Muliere, P. and Secchi, P. (2004) Reinforced urn processes for
  modeling credit default distributions.
\newblock \textit{International Journal of Theoretical and Applied Finance},
  \textbf{7}, 407--423.

\bibitem[{Andersen et~al.(2002)Andersen, Abildstrom and
  Rosth{\o}j}]{Andersen2002}
Andersen, P.~K., Abildstrom, S.~Z. and Rosth{\o}j, S. (2002) Competing risks as
  a multi-state model.
\newblock \textit{Statistical Methods in Medical Research}, \textbf{11},
  203--215.

\bibitem[{Andersen et~al.(2012)Andersen, Borgan, Gill and
  Keiding}]{Andersen2012}
Andersen, P.~K., Borgan, O., Gill, R.~D. and Keiding, N. (2012)
  \textit{Statistical models based on counting processes}.
\newblock New York: Springer Science \& Business Media.

\bibitem[{Antoniak(1974)}]{Antoniak1974}
Antoniak, C.~E. (1974) Mixtures of dirichlet processes with applications to
  {Bayes}ian nonparametric problems.
\newblock \textit{The annals of statistics}, \textbf{2}, 1152--1174.

\bibitem[{Berger and Schmid(2017)}]{Berger2017}
Berger, M. and Schmid, M. (2017) Semiparametric regression for discrete
  time-to-event data.
\newblock \textit{Statistical Modelling}, 1471082X17748084.

\bibitem[{Bernardo and Smith(2000)}]{Bernardo2000}
Bernardo, J.~M. and Smith, A.~F. (2000) \textit{{Bayes}ian theory}.
\newblock John Wiley \& Sons.

\bibitem[{Blackwell and MacQueen(1973)}]{Blackwell1973}
Blackwell, D. and MacQueen, J.~B. (1973) Ferguson distributions via polya urn
  schemes.
\newblock \textit{Ann. Statist.}, \textbf{1}, 353--355.

\bibitem[{Bulla and Muliere(2007)}]{Bulla2007}
Bulla, P. and Muliere, P. (2007) {Bayes}ian nonparametric estimation for
  reinforced markov renewal processes.
\newblock \textit{Statistical Inference for Stochastic Processes}, \textbf{10},
  283--303.

\bibitem[{Bulla et~al.(2009)Bulla, Muliere and Walker}]{Bulla2009}
Bulla, P., Muliere, P. and Walker, S. (2009) A {Bayes}ian nonparametric
  estimator of a multivariate survival function.
\newblock \textit{Journal of Statistical Planning and Inference}, \textbf{139},
  3639--3648.

\bibitem[{Chae et~al.(2013)Chae, Wei{\ss}bach, Cho and Kim}]{Chae2013}
Chae, M., Wei{\ss}bach, R., Cho, K.~H. and Kim, Y. (2013) A mixture of
  beta--dirichlet processes prior for {Bayes}ian analysis of event history
  data.
\newblock \textit{Journal of the Korean Statistical Society}, \textbf{42},
  313--321.

\bibitem[{Cifarelli and Regazzini(1978)}]{Cifarelli1978}
Cifarelli, D.~M. and Regazzini, E. (1978) Problemi statistici non parametrici
  in condizioni di scambiabilità parziale: impiego di medie associative.
\newblock \textit{Tech. rep.}, Quaderni Istituto di Matematica Finanziaria
  dell’Universit\`a di Torino.
\newblock English translation, ``Nonparametric statistical problems under
  partial exchangeability. The role of associative means'', available at
  http://www-dimat.unipv.it/eugenioconference/eugenio.html (last accessed: 24
  March 2018).

\bibitem[{Cifarelli and Regazzini(1996)}]{Cifarelli1996}
--- (1996) {De Finetti's} contribution to probability and statistics.
\newblock \textit{Statistical Science}, \textbf{11}, 253--282.

\bibitem[{Connor and Mosimann(1969)}]{Connor1969}
Connor, R.~J. and Mosimann, J.~E. (1969) Concepts of independence for
  proportions with a generalization of the dirichlet distribution.
\newblock \textit{Journal of the American Statistical Association},
  \textbf{64}, 194--206.

\bibitem[{Coppersmith and Diaconis(1986)}]{Coppersmith1986}
Coppersmith, D. and Diaconis, P. (1986) Random walk with reinforcement.
\newblock Unpublished manuscript.

\bibitem[{Cox(1972)}]{Cox1972}
Cox, D. (1972) Regression models and life-tables.
\newblock \textit{Journal of the Royal Statistical Society. Series B
  (Methodological)}, \textbf{34}, 87--22.

\bibitem[{Crowder(2012)}]{Crowder2012}
Crowder, M.~J. (2012) \textit{Multivariate survival analysis and competing
  risks}.
\newblock Boca Raton, Florida: CRC Press.

\bibitem[{De~Blasi and Hjort(2007)}]{DeBlasi2007}
De~Blasi, P. and Hjort, N.~L. (2007) {Bayes}ian survival analysis in
  proportional hazard models with logistic relative risk.
\newblock \textit{Scandinavian Journal of Statistics}, \textbf{34}, 229--257.

\bibitem[{{de Finetti}(1937)}]{deFinetti1937}
{de Finetti}, B. (1937) La pr{\'{e}}vision: ses lois logiques, ses sources
  subjectives.
\newblock \textit{Annales de l'institut Henri Poincaré}, \textbf{7}, 1--68.
\newblock English translation, ``Foresight: its Logical Laws, Its Subjective
  Sources'', in H. E. Kyburg and H. E. Smokler (editors), Studies in Subjective
  Probability. New York: Wiley, 1964.

\bibitem[{Diaconis(1988)}]{Diaconis1988}
Diaconis, P. (1988) Recent progress on de {F}inetti{\textquoteright}s notions
  of exchangeability.
\newblock In \textit{Bayesian Statistics 3} (eds. J.~Bernardo, M.~DeGroot,
  D.~Lindley and A.~Smith), vol.~3, 111--125. Oxford University Press Oxford,,
  UK.

\bibitem[{Diaconis and Freedman(1980)}]{Diaconis1980}
Diaconis, P. and Freedman, D. (1980) {De Finetti's} theorem for {Markov}
  chains.
\newblock \textit{The Annals of Probability}, \textbf{8}, 115--130.

\bibitem[{Doksum(1974)}]{Doksum1974}
Doksum, K. (1974) Tailfree and neutral random probabilities and their posterior
  distributions.
\newblock \textit{The Annals of Probability}, 183--201.

\bibitem[{Drzewiecki et~al.(1980)Drzewiecki, Ladefoged and
  Christensen}]{Drzewiecki1980}
Drzewiecki, K., Ladefoged, C. and Christensen, H. (1980) Biopsy and prognosis
  for cutaneous malignant melanomas in clinical stage {I}.
\newblock \textit{Scandinavian Journal of Plastic and Reconstructive Surgery},
  \textbf{14}, 141--144.

\bibitem[{Epifani et~al.(2002)Epifani, Fortini and Ladelli}]{Epifani2002}
Epifani, I., Fortini, S. and Ladelli, L. (2002) A characterization for mixtures
  of semi-{Markov} processes.
\newblock \textit{Statistics \& Probability Letters}, \textbf{60}, 445--457.

\bibitem[{Ferguson(1973)}]{Ferguson1973}
Ferguson, T.~S. (1973) A {Bayes}ian analysis of some nonparametric problems.
\newblock \textit{The Annals of Statistics}, \textbf{1}, 209--230.

\bibitem[{Fine(1999)}]{Fine1999}
Fine, J.~P. (1999) Analysing competing risks data with transformation models.
\newblock \textit{Journal of the Royal Statistical Society. Series B,
  Statistical Methodology}, \textbf{61}, 817--830.

\bibitem[{Fine and Gray(1999)}]{Fine1999a}
Fine, J.~P. and Gray, R.~J. (1999) A proportional hazards model for the
  subdistribution of a competing risk.
\newblock \textit{Journal of the American Statistical Association},
  \textbf{94}, 496--509.

\bibitem[{Fortini and Petrone(2012)}]{Fortini2012}
Fortini, S. and Petrone, S. (2012) Predictive construction of priors in
  {Bayes}ian nonparametrics.
\newblock \textit{Brazilian Journal of Probability and Statistics},
  \textbf{26}, 423--449.

\bibitem[{Ge and Chen(2012)}]{Ge2012}
Ge, M. and Chen, M.-H. (2012) {Bayes}ian inference of the fully specified
  subdistribution model for survival data with competing risks.
\newblock \textit{Lifetime Data Analysis}, \textbf{18}, 339--363.

\bibitem[{Gelman et~al.(2013)Gelman, Carlin, Stern, Dunson, Vehtari and
  Rubin}]{Gelman2013}
Gelman, A., Carlin, J., Stern, H., Dunson, D., Vehtari, A. and Rubin, D. (2013)
  \textit{{Bayes}ian Data Analysis}.
\newblock Chapman \& Hall/CRC Texts in Statistical Science. Boca Raton: Taylor
  \& Francis, third edition edn.

\bibitem[{Geweke(1992)}]{Geweke1992}
Geweke, J. (1992) Evaluating the accuracy of sampling-based approaches to
  calcualting posterior moments.
\newblock In \textit{Bayesian Statistics 4} (eds. J.~Bernardo, J.~Berger,
  A.~Dawid and S.~A.F.M.), 169--193. Oxford, UK: Clarendon Press.

\bibitem[{Ghosal and van~der Vaart(2017)}]{Ghosal2017}
Ghosal, S. and van~der Vaart, A. (2017) \textit{Fundamentals of Nonparametric
  {Bayes}ian Inference}.
\newblock Cambridge Series in Statistical and Probabilistic Mathematics.
  Cambridge University Press.

\bibitem[{Guo and Lin(1994)}]{Guo1994}
Guo, S. and Lin, D. (1994) Regression analysis of multivariate grouped survival
  data.
\newblock \textit{Biometrics}, 632--639.

\bibitem[{Heitjan(1993)}]{Heitjan1993}
Heitjan, D.~F. (1993) Ignorability and coarse data: Some biomedical examples.
\newblock \textit{Biometrics}, \textbf{49}, 1099--1109.

\bibitem[{Heitjan and Rubin(1991)}]{Heitjan1991}
Heitjan, D.~F. and Rubin, D.~B. (1991) Ignorability and coarse data.
\newblock \textit{The Annals of Statistics}, \textbf{19}, 2244--2253.

\bibitem[{Hinchliffe and Lambert(2013)}]{Hinchliffe2013}
Hinchliffe, S.~R. and Lambert, P.~C. (2013) Flexible parametric modelling of
  cause-specific hazards to estimate cumulative incidence functions.
\newblock \textit{BMC Medical Research Methodology}, \textbf{13}, 13.

\bibitem[{Hjort(1990)}]{Hjort1990}
Hjort, N.~L. (1990) Nonparametric {Bayes} estimators based on beta processes in
  models for life history data.
\newblock \textit{The Annals of Statistics}, \textbf{18}, 1259--1294.

\bibitem[{Hjort et~al.(2010)Hjort, Holmes, M{\"u}ller and Walker}]{Hjort2010}
Hjort, N.~L., Holmes, C., M{\"u}ller, P. and Walker, S.~G. (2010)
  \textit{{Bayes}ian nonparametrics}.
\newblock Cambridge, UK: Cambridge University Press.

\bibitem[{Jeong and Fine(2007)}]{Jeong2007}
Jeong, J.-H. and Fine, J.~P. (2007) Parametric regression on cumulative
  incidence function.
\newblock \textit{Biostatistics}, \textbf{8}, 184--196.

\bibitem[{Kalbfleisch and Prentice(2002)}]{Kalbfleisch2002}
Kalbfleisch, J.~D. and Prentice, R.~L. (2002) \textit{The statistical analysis
  of failure time data}.
\newblock Hoboken, New Jersey: John Wiley \& Sons, 2\textsuperscript{nd}
  edition edn.

\bibitem[{Kim and Gray(2012)}]{Kim2012}
Kim, H.~T. and Gray, R. (2012) Three-component cure rate model for
  nonproportional hazards alternative in the design of randomized clinical
  trials.
\newblock \textit{Clinical Trials}, \textbf{9}, 155--163.

\bibitem[{Larson and Dinse(1985)}]{Larson1985}
Larson, M.~G. and Dinse, G.~E. (1985) A mixture model for the regression
  analysis of competing risks data.
\newblock \textit{Applied Statistics}, \textbf{34}, 201--211.

\bibitem[{Lawless(2011)}]{Lawless2011}
Lawless, J.~F. (2011) \textit{Statistical models and methods for lifetime
  data}.
\newblock Hoboken, New Jersey: John Wiley \& Sons.

\bibitem[{Lindley and Smith(1972)}]{Lindley1972}
Lindley, D. and Smith, A. (1972) {Bayes} estimates for the linear model.
\newblock \textit{Journal of the Royal Statistical Society. Series B
  (Methodological)}, \textbf{34}, 1--41.

\bibitem[{Martin et~al.(2011)Martin, Quinn and Park}]{Martin2011}
Martin, A.~D., Quinn, K.~M. and Park, J.~H. (2011) {MCMCpack}: {Markov} chain
  {Monte Carlo} in {R}.
\newblock \textit{Journal of Statistical Software}, \textbf{42}, 22.

\bibitem[{Mauldin et~al.(1992)Mauldin, Sudderth and Williams}]{Mauldin1992}
Mauldin, R.~D., Sudderth, W.~D. and Williams, S. (1992) Polya trees and random
  distributions.
\newblock \textit{The Annals of Statistics}, \textbf{20}, 1203--1221.

\bibitem[{Mezzetti et~al.(2007)Mezzetti, Muliere and Bulla}]{Mezzetti2007}
Mezzetti, M., Muliere, P. and Bulla, P. (2007) An application of reinforced urn
  processes to determining maximum tolerated dose.
\newblock \textit{Statistics \& probability letters}, \textbf{77}, 740--747.

\bibitem[{Mira and Petrone(1996)}]{Mira1996}
Mira, A. and Petrone, S. (1996) {Bayes}ian hierarchical nonparametric inference
  for change-point problems.
\newblock In \textit{Bayesian Statistics} (eds. J.~M. Bernardo, J.~Berger,
  A.~Dawid and A.~Smith), no.~5, 693--703. Oxford University Press.

\bibitem[{Muliere et~al.(2006)Muliere, Paganoni and Secchi}]{Muliere2006}
Muliere, P., Paganoni, A.~M. and Secchi, P. (2006) A randomly reinforced urn.
\newblock \textit{Journal of Statistical Planning and Inference}, \textbf{136},
  1853--1874.

\bibitem[{Muliere and Petrone(1993)}]{Muliere1993}
Muliere, P. and Petrone, S. (1993) A {Bayes}ian predictive approach to
  sequential search for an optimal dose: parametric and nonparametric models.
\newblock \textit{Statistical Methods \& Applications}, \textbf{2}, 349--364.

\bibitem[{Muliere et~al.(2000)Muliere, Secchi and Walker}]{Muliere2000}
Muliere, P., Secchi, P. and Walker, S. (2000) Urn schemes and reinforced random
  walks.
\newblock \textit{Stochastic Processes and their Applications}, \textbf{88},
  59--78.

\bibitem[{Muliere et~al.(2003)Muliere, Secchi and Walker}]{Muliere2003}
Muliere, P., Secchi, P. and Walker, S.~G. (2003) Reinforced random processes in
  continuous time.
\newblock \textit{Stochastic Processes and their Applications}, \textbf{104},
  117--130.

\bibitem[{Muliere and Walker(2000)}]{Muliere2000a}
Muliere, P. and Walker, S. (2000) Neutral to the right processes from a
  predictive perspective: a review and new developments.
\newblock \textit{Metron}, \textbf{58}, 13--30.

\bibitem[{M{\"u}ller and Mitra(2013)}]{Mueller2013}
M{\"u}ller, P. and Mitra, R. (2013) {Bayes}ian nonparametric inference--why and
  how.
\newblock \textit{Bayesian Analysis}, \textbf{8}, 269--302.

\bibitem[{Ng et~al.(2011)Ng, Tian and Tang}]{Ng2011}
Ng, K.~W., Tian, G.-L. and Tang, M.-L. (2011) \textit{Dirichlet and related
  distributions: Theory, methods and applications}.
\newblock Chichester, England: John Wiley \& Sons.

\bibitem[{Nieto-Barajas and Walker(2002)}]{Nieto-Barajas2002}
Nieto-Barajas, L.~E. and Walker, S.~G. (2002) {Markov} beta and gamma processes
  for modelling hazard rates.
\newblock \textit{Scandinavian Journal of Statistics}, \textbf{29}, 413--424.

\bibitem[{Orbanz and Roy(2015)}]{Orbanz2015}
Orbanz, P. and Roy, D.~M. (2015) {Bayes}ian models of graphs, arrays and other
  exchangeable random structures.
\newblock \textit{IEEE transactions on pattern analysis and machine
  intelligence}, \textbf{37}, 437--461.

\bibitem[{Peluso et~al.(2015)Peluso, Mira and Muliere}]{Peluso2015}
Peluso, S., Mira, A. and Muliere, P. (2015) Reinforced urn processes for credit
  risk models.
\newblock \textit{Journal of Econometrics}, \textbf{184}, 1--12.

\bibitem[{Pemantle(1988)}]{Pemantle1988}
Pemantle, R. (1988) \textit{Random processes with reinforcement}.
\newblock Ph.D. thesis, Massachussets Institute of Technology.

\bibitem[{Pemantle(2007)}]{Pemantle2007}
--- (2007) A survey of random processes with reinforcement.
\newblock \textit{Probabability Surveys}, \textbf{4}, 1--79.

\bibitem[{Phadia(2015)}]{Phadia2015}
Phadia, E.~G. (2015) \textit{Prior processes and their applications}.
\newblock New York: Springer.

\bibitem[{Pintilie(2006)}]{Pintilie2006}
Pintilie, M. (2006) \textit{Competing risks: a practical perspective}.
\newblock Chichester, England: John Wiley \& Sons.

\bibitem[{Putter et~al.(2007)Putter, Fiocco and Geskus}]{Putter2007}
Putter, H., Fiocco, M. and Geskus, R. (2007) Tutorial in biostatistics:
  competing risks and multi-state models.
\newblock \textit{Statistics in Medicine}, \textbf{26}, 2389--2430.

\bibitem[{Rigat and Muliere(2012)}]{Rigat2012}
Rigat, F. and Muliere, P. (2012) Nonparametric survival regression using the
  {beta-Stacy} process.
\newblock \textit{Journal of Statistical Planning and Inference}, \textbf{142},
  2688--2700.

\bibitem[{Robert and Casella(2004)}]{Robert2004}
Robert, C. and Casella, G. (2004) \textit{{Monte Carlo} Statistical Methods}.
\newblock Springer Texts in Statistics. New York: Springer,
  2\textsuperscript{nd} edn.

\bibitem[{Scheike and Zhang(2011)}]{Scheike2011}
Scheike, T.~H. and Zhang, M.-J. (2011) Analyzing competing risk data using the
  r timereg package.
\newblock \textit{Journal of statistical software}, \textbf{38}.

\bibitem[{Scheike et~al.(2008)Scheike, Zhang and Gerds}]{Scheike2008}
Scheike, T.~H., Zhang, M.-J. and Gerds, T.~A. (2008) Predicting cumulative
  incidence probability by direct binomial regression.
\newblock \textit{Biometrika}, \textbf{95}, 205--220.

\bibitem[{Schmid et~al.(2016)Schmid, K{\"u}chenhoff, Hoerauf and
  Tutz}]{Schmid2016}
Schmid, M., K{\"u}chenhoff, H., Hoerauf, A. and Tutz, G. (2016) A survival tree
  method for the analysis of discrete event times in clinical and
  epidemiological studies.
\newblock \textit{Statistics in medicine}, \textbf{35}, 734--751.

\bibitem[{Singpurwalla(1988)}]{Singpurwalla1988}
Singpurwalla, N.~D. (1988) Foundational issues in reliability and risk
  analysis.
\newblock \textit{SIAM Review}, \textbf{30}, 264--282.

\bibitem[{Sivazlian(1981)}]{Sivazlian1981}
Sivazlian, B. (1981) On a multivariate extension of the gamma and beta
  distributions.
\newblock \textit{SIAM Journal on Applied Mathematics}, \textbf{41}, 205--209.

\bibitem[{Th{\"o}rn et~al.(1994)Th{\"o}rn, Pont{\'e}, Bergstr{\"o}m, Spar{\'e}n
  and Adami}]{Thoern1994}
Th{\"o}rn, M., Pont{\'e}, F., Bergstr{\"o}m, R., Spar{\'e}n, P. and Adami,
  H.-O. (1994) Clinical and histopathologic predictors of survival in patients
  with malignant melanoma: a population-based study in {Sweden}.
\newblock \textit{JNCI: Journal of the National Cancer Institute}, \textbf{86},
  761--769.

\bibitem[{Tutz and Schmid(2016)}]{Tutz2016}
Tutz, G. and Schmid, M. (2016) \textit{Modeling Discrete Time-to-Event Data}.
\newblock Springer Series in Statistics. Springer.

\bibitem[{Walker and Muliere(1997)}]{Walker1997}
Walker, S. and Muliere, P. (1997) {Beta-Stacy} processes and a generalization
  of the p{\'o}lya-urn scheme.
\newblock \textit{The Annals of Statistics}, \textbf{25}, 1762--1780.

\bibitem[{Walker and Muliere(1999)}]{Walker1999}
--- (1999) A characterization of a neutral to the right prior via an extension
  of {Johnson}'s sufficientness postulate.
\newblock \textit{The Annals of Statistics}, \textbf{27}, 589--599.

\bibitem[{Wechsler(1993)}]{Wechsler1993}
Wechsler, S. (1993) Exchangeability and predictivism.
\newblock \textit{Erkenntnis}, \textbf{38}, 343--350.

\bibitem[{Wolbers et~al.(2009)Wolbers, Koller, Witteman and
  Steyerberg}]{Wolbers2009}
Wolbers, M., Koller, M.~T., Witteman, J.~C. and Steyerberg, E.~W. (2009)
  Prognostic models with competing risks: methods and application to coronary
  risk prediction.
\newblock \textit{Epidemiology}, \textbf{20}, 555--561.

\end{thebibliography}
\bibliographystyle{rss}

\end{document}